\documentclass[11pt,reqno]{amsart}
\usepackage{fullpage}
\usepackage{amsmath,amsthm,verbatim,amssymb,amsfonts,amscd, graphicx,bbm,bm}
\usepackage{graphics}
\usepackage{mathtools}
\usepackage{mathrsfs}
\usepackage{svg}
\usepackage{esint}
\usepackage{color}
\usepackage{tikz,tikz-cd}
\usepackage[all,cmtip]{xy}

\mathtoolsset{showonlyrefs}

\usetikzlibrary{decorations.pathreplacing,arrows.meta,calc,patterns,angles,quotes,decorations.markings}

\usepackage{xcolor}
\usepackage[pdfstartview=FitH, bookmarksnumbered=true,bookmarksopen=true, colorlinks=true, pdfborder=001, citecolor=blue, linkcolor=blue,urlcolor=blue]{hyperref}

\newtheorem{theorem}{Theorem}
\newtheorem{proposition}[theorem]{Proposition}
\newtheorem{lemma}[theorem]{Lemma}
\newtheorem{corollary}[theorem]{Corollary}

\theoremstyle{definition}
\newtheorem{remark}{Remark}

\newtheorem{definition}[theorem]{Definition}

\newcommand{\cref}[1]{Corollary~\ref{c.#1}}

\numberwithin{equation}{section}
\numberwithin{theorem}{section}

\title[Viscous bubble resonances with surface tension]{Minnaert resonances and higher-order acoustic modes for bubbles in a viscous fluid with surface tension}

\author[H. Ammari]{Habib Ammari}
\address[H. Ammari]{ETH Z\"urich, Department of Mathematics, Rämistrasse 101, 8092 Z\"urich, Switzerland and 
  Hong Kong Institute for Advanced Study, City University of Hong Kong, Kowloon Tong, Hong Kong Special Administrative Region.}
\email{habib.ammari@math.ethz.ch}

\author[X. Fu]{Xin Fu}
\address[X. Fu]{School of Science, Institute for Theoretical Sciences, Westlake University, Hangzhou, Zhejiang Province, 310030, P.R. China}
\email{1550862645cf@gmail.com}

\author[W. Jing]{Wenjia Jing}
\address[W. Jing]{Yau Mathematical Sciences Center, Tsinghua University, Beijing 100084 and Beijing Institute of Mathematical Sciences and Applications, Beijing 101408, P.R. China}
\email{wjjing@tsinghua.edu.cn}

\author[H. Li]{Hongjie Li}
\address[H. Li]{Yau Mathematical Sciences Center, Tsinghua University, Beijing, China.}
\email{hongjieli@tsinghua.edu.cn; hongjie\_li@yeah.net}


\date{\today}

\begin{document}

\begin{abstract}

We study the low-frequency resonances (Minnaert resonances) for micro-bubbles in a viscous fluid with surface tension. The interior pressure is coupled to the exterior pressure--velocity system through a normal hydrodynamic Neumann-to-Dirichlet operator. We analyze the incompressible limit of this operator and its coupling to the linearized curvature operator. This coupling introduces an additional non-resonance condition needed to isolate the low-frequency acoustic modes considered here. Under this non-resonance condition, we derive leading-order asymptotic formulas for the resonant frequencies in terms of the difference in the acoustic impedance at the interface between the gas and the fluid. 
For multiple bubbles, the collective resonances are governed by an effective capacitance matrix.
The spherical calculation recovers the known viscosity and surface-tension corrections. Possible extensions to higher-order resonances, i.e., beyond the subwavelength regime, are discussed.

\end{abstract}

\maketitle

\tableofcontents

\section{Introduction}

Micro-bubbles are used as ultrasonic contrast agents that enhance the echo from blood due to the large acoustic impedance difference at the interface between gas and blood \cite{review-bubbles,review-contrast-agent,review-contrast-agent-2,outi}. The scattering of ultrasonic waves from the micro-bubbles results in a strong signal at the receiver \cite{transient,review-contrast-agent,microbubble-spectroscopy} due to their resonant frequencies. Contrast-enhanced ultrasound is believed to provide images with the highest quality if it operates at the frequency at which the major part of the micro-bubble population resonates. 

Furthermore, the micro-bubble dynamics can then be tracked by ultra-fast acoustic imaging, allowing non-invasive imaging deep into organs at microscopic scales \cite{tanter1,biomedical}. The micro-bubble behavior can also be characterized optically by insonifying the bubbles and recording their responses with an ultrahigh-speed camera \cite{microbubble-spectroscopy}. In both of these emerging biomedical imaging techniques, an accurate model is needed to analyze the resonant behavior of the micro-bubbles and relate it to their environment. Nevertheless, in practice, an idealized approximation of the resonant behavior of the micro-bubbles is widely used. The bubbles are assumed to oscillate at the undamped natural frequency, corresponding to their lowest-order acoustic mode, which we refer to as the Minnaert frequency. The viscosity of the liquid, the surface tension, and the interactions between the micro-bubbles are neglected, and consequently,  the (measured) frequency at which the signal scattered by a micro-bubble has a resonance peak is not equal to the natural frequency of the micro-bubble. In \cite{khismatullin2004resonance}, it is shown that the resonant frequency of a micro-bubble deviates significantly from the Minnaert resonance due to viscosity. Explicit formulas for the shift of the natural frequency of a spherical micro-bubble due to viscosity and surface tension are given in \cite{khismatullin2004resonance, samek1989multiscale}. We also refer the reader to \cite{add1,add2,add3,add5,add4} for the derivations of the linearized model and oscillations of non-spherical bubbles.

Layer-potential techniques have provided a mathematical framework for the analysis of acoustic resonances in bubbly media. In the absence of viscosity and surface tension, 
Minnaert resonances were analyzed in \cite{ammari2018minnaert}, while capacitance matrix formulations were developed in \cite{ammari2025mathematical,ammari2024functional} to characterize resonance splitting induced by interactions among multiple bubbles. Further developments include the analysis of subwavelength resonances in encapsulated bubbles \cite{ammari2019encapsulated} and higher-order resonances associated with nonzero interior Neumann eigenfrequencies \cite{fabry3d}. Related studies of periodic bubbly media include high-frequency
homogenization near Dirac points in honeycomb crystals \cite{ammari2020high} and subwavelength guided modes associated with line defects \cite{ammari2021subwavelength}.
The present work extends this framework by incorporating viscosity and surface tension into the model and quantifying their effects on the leading-order asymptotics of the resonant frequencies.

This work is also connected to recent mathematical studies of subwavelength resonances in high-contrast media. In \cite{li2022minnaert}, Minnaert resonances were justified for acoustic bubbles embedded in soft elastic materials, where the acoustic oscillation inside the bubble is coupled with elastic wave propagation in the surrounding background. Related work on multi-layer high-contrast acoustic resonators \cite{deng2025multilayer} and on close-to-touching Helmholtz subwavelength resonators \cite{dong2025close} further shows how geometry, material contrast, and resonator separation can influence resonant frequencies and field enhancement. Dipolar, quadrupolar, and hybrid resonant modes generated by hard inclusions embedded in soft elastic materials were mathematically justified and analyzed in \cite{li2026resonant, lxy7007, li2025dipolar}, where stress estimates for a pair of hard inclusions were also established.

The model considered here couples the acoustic pressure inside the bubble with the pressure--velocity system for the surrounding viscous fluid, subject to interface conditions that account for surface tension. We employ scalar layer potentials to represent the interior pressure and hydrodynamic layer potentials to represent the exterior pressure--velocity pair. These representations allow us to establish an equivalence between the resonances of the original transmission problem and the characteristic values of a boundary operator. A key ingredient in this reduction is the normal hydrodynamic Neumann-to-Dirichlet  operator $\mathcal{N}_{k_p,k_s}$, which maps a prescribed normal traction to the normal speed. Unlike the corresponding operators arising in elasticity, the hydrodynamic operator involves compressional and shear wavenumbers $k_p$ and $k_s$ that are generally complex-valued even at real frequencies, owing to viscosity. We establish the solvability, mapping properties, and holomorphic dependence of $\mathcal{N}_{k_p,k_s}$ on $(k_p,k_s)$, as well as the asymptotic behavior as $k_p\to 0$. In particular, the limiting operator $\mathcal{N}_{0,k_s}$ retains the viscous shear response and captures effects that are absent from the inviscid acoustic model.

Surface tension enters the dynamic interface condition through the linearized curvature operator
\[
    \mathcal{H}=-(\Delta_{\partial D}+|B|^2),
\]
where $|B|$ denotes the norm of the second fundamental form of $\partial D$. The resulting coupling between surface tension and the exterior hydrodynamic response gives rise to the boundary operator
\begin{equation}\label{intro:Iminusks2}
    I-k_s^2\gamma\mathcal{N}_{0,k_s}\mathcal{H},
\end{equation}
where $\gamma$ is the dimensionless surface tension coefficient. An important consequence of this coupling is the possible appearance of additional capillary-viscous surface modes. Such modes do not arise in the model without surface tension considered in earlier works \cite{ammari2025mathematical,ammari2024functional,ammari2018minnaert,ammari2019encapsulated,fabry3d}, where $\gamma=0$ and \eqref{intro:Iminusks2} reduces to the identity. To isolate the Minnaert-type resonances, we introduce a non-resonance condition excluding these additional modes and show that this condition holds outside a discrete set of parameters. Under this condition, the kernel of the limiting boundary operator contains only the Minnaert resonant modes with constant interior pressures, allowing one to isolate the Minnaert-type resonance branches from the capillary-viscous surface modes.

Based on the boundary operator reduction, we establish the existence of the Minnaert-type resonances and derive the leading-order asymptotics in terms of the acoustic impedance difference. For multiple bubbles, the leading-order collective resonances are characterized by a capacitance matrix incorporating the viscous and surface-tension effects. For a single spherical bubble, the resulting formulas recover the known corrections in the existing physical literature. Finally, we discuss a possible extension of the present framework to higher-order acoustic modes associated with nonzero interior Neumann eigenfrequencies, building on the recent results in \cite{fabry3d}.
Such higher-order acoustic modes of micro-bubbles have also been observed experimentally \cite{ultrasonic}.

Our paper is organized as follows. Section \ref{sec:1} introduces the linearized transmission problem, the nondimensional formulation, and the parameter regime relevant to Minnaert-type resonances. In Section \ref{sec:2}, we construct the normal hydrodynamic Neumann-to-Dirichlet operator and establish the equivalence between the original resonance problem and a boundary operator formulation. We then analyze the limiting operator and derive the leading-order asymptotic formulas for the Minnaert-type resonances under the non-resonance condition. Section \ref{sec:3} considers the case of a single spherical bubble and recovers known formulas for the effects of viscosity and surface tension. Section \ref{sec:5} discusses higher-order acoustic modes associated with nonzero interior Neumann eigenfrequencies. The appendices provide the derivation of the linearized two-phase flow equations, collect properties of the hydrodynamic layer potentials, and prove some
results of the Gohberg-Sigal theory used in the analysis.

\section{Problem set-up} \label{sec:1}

\subsection{Physical setting} 

We consider wave scattering by a gas bubble immersed in a viscous fluid. Let $\rho_b$, $\kappa_b$, and $c_b$ denote the density, bulk modulus, and sound speed of the gas, respectively, and let $\rho$, $\kappa$, and $c$ denote the corresponding quantities of the fluid. Then
\begin{equation}
    \kappa_b=\rho_b c_b^2,\qquad \kappa=\rho c^2.
\end{equation}
Let $\mu$ and $\lambda$ denote the dynamic viscosity and the second viscosity of the fluid, respectively, satisfying \cite[Eq.~(15.4)]{landau1987fluid}
\begin{equation}
    \mu>0,\qquad 3\lambda+2\mu\geq 0.
\end{equation}
Let $D\subset\mathbb{R}^3$ be a bounded smooth domain representing the bubble at equilibrium, with a connected exterior $\mathbb{R}^3\setminus\overline{D}$. Under the time-harmonic convention $\mathrm{e}^{\mathrm{i}\omega t}$, the scattering problem takes the form
\begin{equation}\label{maineq}
\left\{
\begin{array}{@{}l@{}}
\begin{alignedat}{3}
    &\mathrm{i}\omega p + \kappa_b \nabla \cdot \mathbf{v} =0,
    &\qquad& \mathrm{i}\omega \rho_b \mathbf{v} = -\nabla p
    &\qquad& \mathrm{in}\ D, \\
    &\mathrm{i}\omega p + \kappa \nabla \cdot \mathbf{v} =0,
    &\qquad& \mathrm{i}\omega \rho \mathbf{v}
    = \nabla \cdot \mathbf{T}(p,\mathbf{v})
    &\qquad& \mathrm{in}\ \mathbb{R}^3 \setminus \overline{D}, \\
    &\mathbf{v}\cdot \mathbf{n}|_- = \mathbf{v}\cdot \mathbf{n}|_+,
    &\qquad& p \mathbf{n}|_-+ \mathbf{T}(p,\mathbf{v})\mathbf{n}|_+
    = \frac{\sigma}{\mathrm{i}\omega}
    \mathcal{H}(\mathbf{v}\cdot \mathbf{n})\mathbf{n}
    &\qquad& \mathrm{on}\ \partial D,
\end{alignedat}\\[1mm]
\mathbf{v}\ \text{satisfies an outgoing radiation condition}.
\end{array}
\right.
\end{equation}
Here and throughout, $\mathbf{n}$ is the outward unit normal to $D$. The symbols $|_+$ and $|_-$ denote the traces taken from the exterior and interior of $D$, respectively. Moreover,
\begin{equation}
    \mathbf{T}(p,\mathbf{v}) := - p \mathbf{I} + 2\mu \nabla^{\mathrm{s}} \mathbf{v}  + \lambda (\nabla \cdot \mathbf{v}) \mathbf{I}, \qquad \mathrm{where} \  \nabla^{\mathrm{s}} \mathbf{v} = \frac{1}{2}\big(\nabla \mathbf{v} + \nabla \mathbf{v}^\top  \big),
\end{equation}
is the Cauchy stress tensor of the fluid. The surface tension coefficient is $\sigma>0$, and the linearized curvature operator $\mathcal{H}$ is defined by
\begin{equation}\label{Hf}
    \mathcal{H} := -(\Delta_{\partial D} + |B|^2) ,
\end{equation}
where $\Delta_{\partial D}$ denotes the Laplace-Beltrami operator on $\partial D$, and $|B|$ denotes the norm of the second fundamental form of $\partial D$; see \eqref{defofB2}.
The outgoing radiation condition is specified in Section \ref{secradiationcind} after nondimensionalization. 

A derivation of \eqref{maineq} from the underlying two-phase flow model is given in Appendix \ref{appenequ}.

\subsection{Nondimensionalization and the mathematical formulation}

Let $a$ be the typical size (for instance, the diameter) of the domain $D$ and set
\begin{equation}\label{rescaleD}
    x=ax', \qquad D=aD'.
\end{equation}
We introduce dimensionless velocity and pressure as
\begin{equation}\label{nondimunknown}
    \mathbf{v} (x)= V \mathbf{v}'(x'), \qquad p(x) = Pp'(x'),
\end{equation}
where $V>0$ is a reference velocity and the reference pressure $P$ is chosen by
\begin{equation}
    P=\frac{\rho_b \kappa}{\mathrm{i}\omega \rho a} V.
\end{equation}

Define the dimensionless viscosity coefficients
\begin{equation}\label{nondmulambda}
    \tilde{\mu} :=\frac{\mathrm{i}\omega \rho}{\rho_b \kappa} \mu, \qquad \tilde \lambda :=\frac{\mathrm{i}\omega \rho}{\rho_b \kappa} \lambda,
\end{equation}
and the Cauchy tensor $\widetilde{\mathbf{T}} (p',\mathbf{v}') := - p' \mathbf{I} + 2\tilde{\mu}\nabla_{x'}^s \mathbf{v}'  + \tilde{\lambda} (\nabla_{x'} \cdot \mathbf{v}') \mathbf{I}$. It can be checked that
\begin{equation}\label{rescaleT}
    \mathbf{T}(p,\mathbf{v}) = P \widetilde{\mathbf{T}} (p',\mathbf{v}').
\end{equation}

Let $|B'|$ be the norm of the second fundamental form of $\partial D'$. Since $|B'|^2 = a^2|B|^2$,
\begin{equation}\label{rescaleH}
    \mathcal{H}(\mathbf{v}\cdot \mathbf{n})  = \frac{V}{a^2} \mathcal{H}'(\mathbf{v}'\cdot \mathbf{n}),\qquad \mathrm{where} \  \mathcal{H}':=- (\Delta_{\partial D'} + |B'|^2).
\end{equation}

Substituting these relations into \eqref{maineq} and \emph{dropping the primes}, we obtain the dimensionless system
\begin{equation}\label{ndmaineqnnn}\left\{
\begin{array}{@{}l@{}}
\begin{alignedat}{3}
    &\tau^2 p + \nabla\cdot \mathbf{v} =0, && \qquad k^2 \mathbf{v} = \nabla p && \quad\mathrm{in}\ D , \\
    &\delta p + \nabla\cdot \mathbf{v} =0, && \qquad k^2 \mathbf{v} + \delta\nabla\cdot \widetilde{\mathbf{T}} (p,\mathbf{v}) =0 &&\quad\mathrm{in}\ \mathbb{R}^3 \setminus \overline{D} , \\
    & \mathbf{v}\cdot \mathbf{n}|_- = \mathbf{v}\cdot \mathbf{n}|_+, &&\qquad p \mathbf{n}|_-+ \widetilde{\mathbf{T}}(p,\mathbf{v})\mathbf{n}|_+  = \gamma \frac{k^2}{\delta} \mathcal{H}(\mathbf{v}\cdot \mathbf{n}) \mathbf{n}  && \quad\mathrm{on}\ \partial D, \\
\end{alignedat}\\[1mm]
     \mathbf{v} \textrm{ satisfies the outgoing radiation condition }\eqref{radiationcond}  ,
    \end{array}\right.
\end{equation}
where the dimensionless parameters are 
\begin{equation}\label{nondtauk}
    \tau := \sqrt{\frac{\rho_b \kappa}{\rho \kappa_b}}, \qquad k:=\sqrt{\frac{\rho }{\kappa} } \omega a, \qquad \delta:= \frac{\rho_b}{\rho}, \qquad \gamma:=\frac{\sigma }{\rho \omega^2 a^3}.
\end{equation}

For convenience in subsequent analysis, we introduce the squared dimensionless compressional and shear wavenumbers 
\begin{equation}\label{defkpks}
    k_p^2 :=  \frac{k^2}{1+\delta(\tilde\lambda +2\tilde\mu)}, \qquad k_s^2 :=  \frac{k^2}{\delta\tilde\mu},
\end{equation}
and the modified pressure
\begin{equation}\label{definitionq}
    q:= \frac{1}{\tilde\mu} (p- \tilde \lambda \nabla \cdot \mathbf{v}) \qquad \mathrm{in} \ \mathbb{R}^3 \setminus \overline{D}.
\end{equation}
The modified pressure $q$ absorbs the second-viscosity term into the pressure, so that the Cauchy tensor takes the form $\widetilde{\mathbf T} =\tilde\mu(2\nabla^{\mathrm s}\mathbf v-q\mathbf I)$.

With the substitutions \eqref{defkpks}-\eqref{definitionq}, and the notations
\begin{equation}
\left\{
\begin{aligned}
    & \mathscr{L}_{k_s}(q,\mathbf{v})=\nabla \cdot (2\nabla^{\mathrm{s}} \mathbf{v} - q \mathbf{I}) + k_s^2 \mathbf{v}, \qquad 
    \mathscr{D}_{k_p,k_s}(q,\mathbf{v})= \nabla \cdot \mathbf{v} - \frac{k_p^2}{2k_p^2 - k_s^2} q, \\
    & \mathbf t_{\mathbf n}(q,\mathbf{v})= (2\nabla^{\mathrm s}\mathbf v-q\mathbf I)\mathbf n|_+,
\end{aligned}\right.
\end{equation}
we conclude from \eqref{ndmaineqnnn} that
\begin{equation}\label{nondfinal1}
    \left\{
\begin{aligned}
    &(\Delta + k^2\tau^2) p = 0 && \mathrm{in}\ D , \\
    & \mathscr{L}_{k_s}(q,\mathbf{v})=0, \quad \mathscr{D}_{k_p,k_s}(q,\mathbf{v})=0  &&\mathrm{in}\ \mathbb{R}^3 \setminus \overline{D} , \\
    &  \mathbf{v}\cdot \mathbf{n}|_+ =\left.\frac{1}{k^2}\frac{\partial p}{\partial \mathbf{n}}\right|_-  && \mathrm{on}\ \partial D, \\
    & \mathbf t_{\mathbf n}(q,\mathbf{v})= k_s^2 \left[ \gamma \mathcal{H}\left(\frac{1}{k^2}\left. \frac{\partial p}{\partial \mathbf{n}}\right|_- \right) -\frac{\delta}{k^2}  p \big|_-  \right] \mathbf{n}   && \mathrm{on}\ \partial D, \\
    &  \mathbf{v} \textrm{ satisfies the outgoing radiation condition }\eqref{radiationcond} .\\
    \end{aligned}\right.
\end{equation}
Here $p$ is defined in $D$, while $(q,\mathbf{v})$ is defined in $\mathbb{R}^3\setminus \overline{D}$.

\subsection{Minnaert-type resonance regime}

Throughout, for complex quantities $z,w$, we write $z\asymp w$ if there are universal constants $c,\,C>0$ such that $c|w| < |z| <C|w|$.

We consider the regime
\begin{equation}\label{lowdensity}
    \tau\asymp 1, \qquad k\ll 1, \qquad \delta \ll 1,
\end{equation}
which is common in the study of Minnaert resonances; see \cite{ammari2018minnaert}. To capture the effects of viscosity and surface tension, we further assume that
\begin{equation}\label{regimelongwave}
     k_p \ll k_s \asymp 1, \qquad \gamma \asymp 1.
\end{equation}
This assumption is physically motivated as follows:
\begin{itemize}
    \item[(a)] The scaling $k_s \asymp  1$ indicates that the \emph{viscous damping factor} 
    \begin{equation}
\beta_{\rm vis}=\frac{2\mu}{\rho a^2}
\end{equation}
\cite[Eq.~(91)]{review-bubbles} is of the same order as the oscillation frequency, as seen from
\begin{equation}
\frac{\beta_{\rm vis}}{\omega}=\frac{2\mu}{\rho a^2\omega}\asymp \frac{1}{|k_s|^2} \asymp 1.
\end{equation}
As observed by Khismatullin \cite{khismatullin2004resonance}, the resonant frequency of the microbubbles can deviate significantly from the undamped Minnaert frequency as the viscous damping increases;

    \item[(b)] The condition $k_p\ll 1$ ensures that the exterior compressional wavelength is much larger than the bubble size, consistent with the long-wave approximation underlying the classical Minnaert resonance. Moreover, the scale separation $k_p\ll k_s$ is common in bubble acoustics (see \cite[the discussion following Eq.~(10)]{review-bubbles}). In particular, it is satisfied for ordinary liquids, such as water or blood plasma, at audible and ultrasonic frequencies \cite{khismatullin2004resonance,10.1063/1.3043575}; 

    \item[(c)] The scaling $\gamma \asymp 1$ allows surface tension to affect the Minnaert resonance at leading order. Indeed, the interface conditions in \eqref{nondfinal1} show that the surface-tension term is weighted by $k_s^2\gamma$. To retain this contribution at leading order, we require $k_s^2\gamma\asymp 1$, which reduces to $\gamma\asymp 1$ under the assumption $k_s\asymp 1$.
\end{itemize}

As an important example, consider an air bubble in blood plasma with typical size $a=1\,\mu\mathrm{m}$ at an ultrasound frequency $f=1\,\mathrm{MHz}$. Using the representative plasma parameters from \cite{khismatullin2004resonance},
\[
\begin{aligned}
\rho &= 1027\,\mathrm{kg/m^3},
& c &= 1543\,\mathrm{m/s},\\
\mu &= 1.2\times10^{-3}\,\mathrm{Pa\cdot s},
& \sigma &= 4.0\times10^{-2}\,\mathrm{N/m},
\end{aligned}
\]
together with $\rho_b=1.2\,\mathrm{kg/m^3}$ and $c_b=343\,\mathrm{m/s}$ for air, we obtain
\[
\begin{aligned}
\tau &\approx 4.50,
& |k_p| \leq k &\approx 4.07\times10^{-3},
& \delta &\approx 1.17\times10^{-3},\\
|k_s| &\approx 2.32,
& \gamma &\approx 0.987.
\end{aligned}
\]
These are consistent with the regime \eqref{lowdensity}--\eqref{regimelongwave}.

\subsection{Outgoing radiation conditions}\label{secradiationcind} 

We first recall the outgoing radiation condition for the Helmholtz equation. Throughout, $B_R$ denotes the open ball of radius $R>0$ centered at the origin in $\mathbb{R}^3$, and $\mathbb{S}^2$ denotes the unit sphere in $\mathbb{R}^3$. We also write $\mathbb{C}^*:=\mathbb C\setminus\{0\}$.

For $R>0$ and $z\in \mathbb{C}$, consider a scalar solution of
\begin{equation}\label{scalareqatz}
    (\Delta+z^2) u=0 \qquad \mathrm{in} \ \mathbb{R}^3\setminus \overline{B_R}.
\end{equation}
We say that $u$ is \emph{outgoing at} $z$ (see \cite[§2.14.5]{ammari2018mathematical}) if 
\begin{itemize}
    \item[(a)] for $z\in \mathbb{C}^*$, $u$ admits the expansion
    \begin{equation}\label{outgoingHelmholtz}
        u(r\theta)=\sum_{\ell=0}^{\infty} \sum_{m=-\ell}^{\ell} a_{\ell m}h_\ell^{(1)}(zr)
    Y_{\ell m}(\theta), \qquad r>R,\ \theta \in \mathbb{S}^2,
    \end{equation}
    where $h_\ell^{(1)}$ is the spherical Hankel function of the first kind of order $\ell$, and $\{Y_{\ell m}\}$ are the spherical harmonics on $\mathbb{S}^2$;

    \item[(b)] for $z=0$, $u(r\theta)=O(r^{-1})$ as $r\rightarrow \infty$, uniformly in $\theta\in\mathbb{S}^2$.
\end{itemize}
A vector-valued solution $\mathbf{u}$ of \eqref{scalareqatz} is called outgoing at $z$, if each component is outgoing at $z$.

\begin{remark}\label{remarksrc}
    For $z\in \mathbb{C}^*$ with $\operatorname{Im}z\geq 0$, the outgoing radiation condition above is equivalent to the Sommerfeld radiation condition
\[
    (\partial_r-\mathrm{i}z)u(r\theta)=O(r^{-2}) \qquad \text{as }r\to\infty,
\]
uniformly in $\theta\in\mathbb{S}^2$; see \cite[§2.8.2]{ammari2018mathematical}.
\end{remark}

Recall that $k_p^2$ and $k_s^2$ are defined in \eqref{defkpks}. At real frequencies $\omega>0$, we choose $k_p$ and $k_s$ as the square roots with positive imaginary parts. For complex $\omega$, $k_p$ and $k_s$ are determined by analytic continuation of the above choice on a specified branch.

We now formulate the outgoing radiation condition for the exterior pressure--velocity system in \eqref{nondfinal1}. Following the standard compressional--shear decomposition (see \cite[§2.15]{ammari2018mathematical}), we write the exterior velocity as $\mathbf{v}=\mathbf{v}_p+\mathbf{v}_s$, where
\begin{equation}\label{defvpvs}
\left\{
    \begin{aligned}
        &\mathbf{v}_p:= \frac{1}{k_s^2-k_p^2} (\Delta + k_s^2)\mathbf{v} = \frac{1}{k_s^2-2k_p^2}  \nabla q, \\
        &\mathbf{v}_s:= \frac{1}{k_p^2-k_s^2} (\Delta + k_p^2)\mathbf{v} .
    \end{aligned}\right.
\end{equation}
The regime $k_p\ll k_s\asymp 1$ in \eqref{regimelongwave} ensures that the denominators in \eqref{defvpvs} are nonzero. Moreover, it follows from \eqref{nondfinal1} that
\begin{equation}\label{hemlholzvpvs}
(\Delta + k_p^2)q =0,  \qquad \mathrm{and} \qquad (\Delta + k_s^2)\mathbf{v}_s =0, \quad \nabla \cdot \mathbf{v}_s =0.
\end{equation}
Accordingly, we impose the following outgoing radiation condition:
\begin{equation}\label{radiationcond}
q \textrm{ and }\mathbf{v}_s \textrm{ are outgoing at wavenumbers }k_p \textrm{ and }k_s \textrm{, respectively.}
\end{equation}

\section{Analysis of the Minnaert-type resonances} \label{sec:2}

As discussed immediately after Remark \ref{remarksrc}, the compressional and shear wavenumbers $k_p$ and $k_s$ holomorphically depend on the physical frequency $\omega$. However, for the spectral analysis below, we treat $k_p$ and $k_s$ as independent complex variables, satisfying $k_p\in\mathbb{C}$, $k_s\in\mathbb{C}^*$, and
\begin{equation}\label{assumpkpks}
    |k_p|< \frac{|k_s|}{\sqrt{2}}.
\end{equation}
In particular, $k_p^2 \neq k_s^2$ and $2k_p^2 \neq k_s^2$, ensuring that the compressional--shear decomposition in \eqref{defvpvs} is well-defined.

\subsection{Normal hydrodynamic Neumann-to-Dirichlet operators}

We seek a boundary operator formulation of \eqref{nondfinal1}. The interior pressure $p$ can be represented using scalar layer potentials, while the exterior pressure--velocity pair $(q,\mathbf{v})$ can be represented using hydrodynamic layer potentials. To couple these representations through the interface conditions, we introduce the \emph{normal hydrodynamic Neumann-to-Dirichlet} (nhNtD) operator and establish the solvability, mapping, and asymptotic properties needed for later analysis.

For $g\in C^\infty(\partial D)$, consider the exterior traction problem
\begin{equation}\label{uindefDtN}
    \left\{
    \begin{aligned}
        & \mathscr{L}_{k_s}(q,\mathbf{v})=0, \quad \mathscr{D}_{k_p,k_s}(q,\mathbf{v})=0  &&\mathrm{in}\ \mathbb{R}^3 \setminus \overline{D} , \\
    & \mathbf t_{\mathbf n}(q,\mathbf{v}) = g\mathbf{n}  && \mathrm{on}\ \partial D, \\
        & \mathbf{v} \textrm{ satisfies the outgoing radiation condition }\eqref{radiationcond} .
    \end{aligned}\right.
\end{equation}
The equations and traction condition in \eqref{uindefDtN} form an elliptic boundary value problem
for the Stokes system when $k_p=0$, and for a Lam\'e-type system after eliminating $q$ when $k_p\neq0$. Therefore, standard boundary elliptic regularity implies that any weak solution of \eqref{uindefDtN} is a classical smooth solution.

\begin{definition}\label{definitionnhNtD}
Given $g\in C^\infty(\partial D)$, whenever \eqref{uindefDtN} admits a unique solution $(q,\mathbf{v})$, the nhNtD operator $\mathcal{N}_{k_p,k_s}$ is defined by
    \begin{equation}
        \mathcal{N}_{k_p,k_s}[g]=\mathbf{n}\cdot \mathbf{v}|_+ .
    \end{equation}
\end{definition}

\begin{remark}\label{rem:generalized-solutions}
The preceding definition is for smooth data $g\in C^\infty(\partial D)$. To accommodate rough boundary data, we formulate \eqref{uindefDtN} in a generalized sense. All distributional duality pairings
below are complex bilinear.

For $g\in\mathcal D'(\partial D)$, a distributional solution of \eqref{uindefDtN} consists of a triple
\[
    (q,\mathbf v,\mathbf f)\in\mathcal D'(\mathbb R^3\setminus\overline D)\times\mathcal D'(\mathbb R^3\setminus\overline D;\mathbb C^3)
    \times\mathcal D'(\partial D;\mathbb C^3),
\]
where $(q,\mathbf v)$ satisfies the distributional equation
\begin{equation}\label{distributionsolutionqv}
    \begin{cases}
        \mathscr{L}_{k_s}(q,\mathbf{v})=0, \quad \mathscr{D}_{k_p,k_s}(q,\mathbf{v)}=0, \qquad \mathrm{in} \ \mathcal{D}'(\mathbb  R^3\setminus\overline D), \\
         \mathbf{v} \textrm{ satisfies the outgoing radiation condition }\eqref{radiationcond}.
    \end{cases}
\end{equation}
By the discussion after \eqref{uindefDtN} and interior elliptic regularity, $q$ and $\mathbf v$ are
smooth in $\mathbb R^3\setminus\overline D$. In particular, the outgoing radiation condition
\eqref{radiationcond} is well-defined in the classical sense.

The boundary condition enters through the identity
\begin{equation}\label{distributionboundarycondition}
\begin{aligned}
    &\langle\mathbf v,\mathscr L_{k_s}(\pi,\mathbf w)\rangle_{\mathbb R^3\setminus\overline D}+\langle q,\mathscr D_{k_p,k_s}(\pi,\mathbf w)\rangle_{\mathbb R^3\setminus\overline D}\\
    &\qquad=\langle g,\mathbf n\cdot\mathbf w|_{\partial D}\rangle_{\partial D}
    -\langle\mathbf f,\mathbf t_{\mathbf n}(\pi,\mathbf w)\rangle_{\partial D}
\end{aligned}
\end{equation}
for every $(\pi,\mathbf w)\in C_0^\infty(\mathbb R^3) \times C_0^\infty(\mathbb R^3;\mathbb C^3)$ satisfying
\begin{equation}
    \mathscr L_{k_s}(\pi,\mathbf w)=0 \quad \mathrm{and} \quad \mathscr D_{k_p,k_s}(\pi,\mathbf w)=0
\end{equation}
in an exterior neighborhood of $\partial D$. 

For $g\in C^\infty(\partial D)$, \eqref{distributionsolutionqv}--\eqref{distributionboundarycondition} is equivalent to \eqref{uindefDtN} with $\mathbf{f} = \mathbf{v}|_+$. We therefore refer to $\mathbf{f}$ as the \emph{generalized exterior velocity trace} and write $\mathbf{f}=\mathbf v|_+$.
Whenever the generalized problem \eqref{distributionsolutionqv}--\eqref{distributionboundarycondition} is uniquely solvable, we extend the definition of the nhNtD operator by
\[
    \mathcal N_{k_p,k_s}[g] :=\mathbf n\cdot\mathbf f\in\mathcal D'(\partial D).
\]
\end{remark}

We now establish some useful properties of the nhNtD operators,
beginning with the uniqueness of \eqref{uindefDtN} when $k_p=0$ and $\operatorname{Im}k_s>0$.

\begin{lemma}\label{uniquekpzero}
    For $k_p=0$, $\operatorname{Im}k_s>0$ and $g=0$, the only solution $(q,\mathbf{v})$ to \eqref{uindefDtN} is zero.
\end{lemma}
\begin{proof}
    Choose $R>0$ such that $D\Subset B_R$. Multiplying the equation $\mathscr{L}_{k_s}(q,\mathbf{v})=0$ with $\overline{\mathbf{v}}$, integrating over $B_R\setminus D$, and using
    \begin{equation}
        \mathscr D_{0,k_s}(q,\mathbf{v}) = \nabla \cdot \mathbf{v}=0
    \end{equation}
    and $\mathbf t_{\mathbf n}(q,\mathbf{v}) =0$, we get
\begin{equation}\label{identitylimit1}
    \int_{B_R\setminus D} 2|\nabla^{\mathrm{s}} \mathbf{v}|^2 - k_s^2 |\mathbf{v}|^2 \,dx = \int_{\partial B_R} (2\nabla^{\mathrm{s}} \mathbf{v} - q \mathbf{I}) \mathbf{n}_R \cdot \overline{\mathbf{v}}\,dS ,
\end{equation}
where $\mathbf n_R$ is the outward unit normal to $B_R$. 

Since $k_p=0$, $q$ is harmonic by \eqref{hemlholzvpvs} and $q(x)=O(r^{-1})$ as $r=|x|\rightarrow \infty$ by \eqref{radiationcond}. The interior elliptic estimates yield
\begin{equation}
    \nabla^m q(x) = O(r^{-m-1}) \qquad \mathrm{as} \ r=|x|\rightarrow \infty,
\end{equation}
for every $m\in \mathbb{N}$. Moreover, it follows from the expansion \eqref{outgoingHelmholtz} and $\operatorname{Im}k_s>0$ that $\mathbf{v}_s$ and $\nabla \mathbf{v}_s$ decay exponentially at infinity. Using the compressional-shear decomposition in \eqref{defvpvs}, we conclude that
\begin{equation}\label{infinityestimate1}
    \int_{\partial B_R} (2\nabla^{\mathrm{s}} \mathbf{v} - q \mathbf{I}) \mathbf{n} \cdot \overline{\mathbf{v}}\,dS  =   O(R^{-1}) \qquad \mathrm{as} \ R\rightarrow \infty.
\end{equation}
Letting $R\to\infty$ in \eqref{identitylimit1} yields
\begin{equation}
    \int_{\mathbb{R}^3\setminus D} 2|\nabla^{\mathrm{s}} \mathbf{v}|^2 - k_s^2 |\mathbf{v}|^2 \,dx =0.
\end{equation}
It follows from $\operatorname{Im}k_s>0$ that $k_s^2 \in \mathbb{C}\setminus [0,\infty)$ and hence $\mathbf{v} = 0$. The equation $\mathscr{L}_{k_s}(q,\mathbf{v})=0$ then gives $\nabla q=0$. Since $\mathbb{R}^3\setminus \overline{D}$ is connected and $q$ decays to zero at infinity, $q=0$. This completes the proof.
\end{proof}

We next establish the solvability of \eqref{uindefDtN} for rough boundary data using hydrodynamic layer potentials.

Suppose that
\begin{equation}
    D= \bigcup_{j=1}^N D_j,
\end{equation}
where $D_1,\cdots,D_N$ are the connected components of $D$. For each $j$, fix a point $x_j\in D_j$ and define the monopole fields
\begin{equation}\label{defofmonopolemode}
q^*_{j,k_p} := G^{k_p}(\cdot-x_j) , \quad \mathbf{v}^*_{j,k_p,k_s} := \frac{1}{k_s^2 - 2 k_p^2} \nabla G^{k_p}(\cdot -x_j)  \qquad \mathrm{in}\  \mathbb{R}^3\setminus \overline{D} .
\end{equation}
Let
\begin{equation}
    \mathbf{t}_j^{k_p,k_s} := \mathbf{t}_{\mathbf{n}}(q^*_{j,k_p}, \mathbf{v}^*_{j,k_p,k_s})
\end{equation}
denote the corresponding exterior traction. 

For $s\in \mathbb{R}$, introduce the operator 
\begin{equation}\label{def:b}
\begin{aligned}
    \mathcal{B}_{k_p,k_s}&:H^s(\partial D; \mathbb{C}^3) \times \mathbb{C}^N \rightarrow H^s(\partial D; \mathbb{C}^3) \times \mathbb{C}^N ,\\
    &\begin{pmatrix}
    \bm{\varphi} \\
    \bm c
\end{pmatrix}
    \mapsto \begin{pmatrix}
          ( \frac{1}{2}\mathbf{I} + \mathbf{K}^{k_p,k_s,*}  ) \bm{\varphi} +\sum_{j=1}^N c_j \mathbf{t}_j^{k_p,k_s} 
           \\
        \int_{\partial D_1} \bm{\varphi}\cdot \mathbf{n}\,dS \\
        \cdots \\
        \int_{\partial D_N} \bm{\varphi}\cdot \mathbf{n}\,dS
    \end{pmatrix} , 
\end{aligned}
\end{equation}
where $\mathbf{K}^{k_p,k_s,*}$ is the hydrodynamic Neumann-Poincar\'{e} operator defined in Appendix \ref{appenlayerpotential}, and $\bm c=(c_1,\cdots,c_N)^\top$. When $s<0$, the integrals above are understood in the distributional sense.

\begin{lemma}\label{lemma34Bkpks}
For $s\in\mathbb R$, denote $X_s=H^s(\partial D;\mathbb C^3)\times\mathbb C^N$, and
\[
    \mathcal U= \left\{ (k_p,k_s)\in\mathbb C\times \mathbb{C}^*: |k_p|<\frac{|k_s|}{\sqrt2}\right\}.
\]
The following statements hold:
\begin{itemize}
    \item[(a)] The map $(k_p,k_s)\mapsto\mathcal B_{k_p,k_s}$ is holomorphic on $\mathcal U$ with values in  $\mathcal{L}(X_s)$. 
    \item[(b)] The set
    \begin{equation}\label{defofEformula}
    \mathcal E:=\{k_s\in\mathbb C^*:\mathcal B_{0,k_s}\text{ is not invertible on }X_0\}
    \end{equation}
is closed and discrete in $\mathbb C^*$ and contains no point in the upper half-plane $\{z\in \mathbb{C}:\operatorname{Im}z>0\}$. 
    \item[(c)] For every $s\in \mathbb{R}$ and every compact set $\mathcal C\subset \mathbb C^*\setminus \mathcal E$, there exists $\varepsilon=\varepsilon(\mathcal{C},s,D)>0$ such that
    $\mathcal B_{k_p,k_s}:X_s\rightarrow X_s$ is invertible and the inverse is holomorphic near
\begin{equation}\label{defofVeps}
    \mathcal{V}=\mathcal{V}(s,\mathcal{C}):= \{z\in \mathbb{C}:|z|<\varepsilon\} \times \mathcal{C} \subset \mathcal{U}.
\end{equation}
Moreover, for $|k_p|<\varepsilon$,
    \begin{equation}\label{Bsmallkp}
        \mathcal B_{k_p,k_s}^{-1}=\mathcal B_{0,k_s}^{-1}+\mathcal{O}_{\mathcal L(X_s)}(|k_p|),
    \end{equation}
    uniformly for $k_s\in\mathcal C$.
\end{itemize}
\end{lemma}

\begin{proof}
For brevity, write $H^s=H^s(\partial D;\mathbb C^3)$. We write $\mathcal B_{k_p,k_s}$ as
\[
    \mathcal B_{k_p,k_s}=\begin{pmatrix}
        \frac12\mathbf I+\mathbf K^{k_p,k_s,*}&\mathcal F_{k_p,k_s}\\
        \mathcal L& 0
    \end{pmatrix},
\]
where
\begin{equation}\label{defofFandL}
    \mathcal F_{k_p,k_s}\bm c:=\sum_{j=1}^N c_j\mathbf t_j^{k_p,k_s},\qquad
    \mathcal L\bm\varphi:=\left(\int_{\partial D_i}\bm\varphi\cdot\mathbf n\,dS\right)_{i=1}^N.
\end{equation}

Since $x_j\in D_j$, each traction $\mathbf t_j^{k_p,k_s}$ is smooth on $\partial D$
and depends holomorphically on $(k_p,k_s)\in \mathcal{U}$. Thus, $\mathcal F_{k_p,k_s}$ is holomorphic on $\mathcal{U}$ with values in $\mathcal{L}(\mathbb{C}^N,H^s)$. By Lemma \ref{lemmaKmaping}(a), $\mathbf K^{k_p,k_s,*}$ is holomorphic on $\mathcal U$ with values in $\mathcal{L}(H^s)$. Hence,
$\mathcal B_{k_p,k_s}$ is holomorphic on $\mathcal{U}$ with values in $\mathcal L(X_s)$. This proves (a).

We next prove (b). At $k_p=0$, $\mathbf K^{0,k_s,*}$ maps $H^s$ into $H^{s+1}$ by Lemma \ref{lemmaKmaping}(a). So, $\mathbf K^{0,k_s,*}$ is compact on $H^s$. Since $\mathcal F_{0,k_s}$ and $\mathcal L$ are finite-rank operators, $\mathcal B_{0,k_s}$ is a compact perturbation of the invertible operator $\operatorname{diag}(\frac12 \mathbf{I},I_N)$.
Thus, $\mathcal B_{0,k_s}:X_s\rightarrow X_s$ is Fredholm of index zero. The analytic Fredholm theorem \cite[Theorem C.8]{dyatlov2019mathematical} applied to 
\begin{equation}
    k_s\mapsto\mathcal B_{0,k_s}:X_0\rightarrow X_0
\end{equation}
now shows that $\mathcal E$ defined in \eqref{defofEformula} is closed and discrete in $\mathbb C^*$.

For every $(\bm\varphi,\bm c)\in\ker\mathcal B_{0,k_s}$, 
\[
    \bm\varphi =-2\mathbf K^{0,k_s,*}\bm\varphi-2\mathcal F_{0,k_s}\bm c.
\]
The smoothing property of $\mathbf K^{0,k_s,*}$ and the smoothness of the monopole tractions imply that $\boldsymbol\varphi\in C^\infty(\partial D;\mathbb C^3)$. Since $\mathcal B_{0,k_s}:X_s\rightarrow X_s$ is Fredholm of index zero, this implies that
\begin{equation}
    \mathcal{E}=\{k_s\in\mathbb C^*:\mathcal B_{0,k_s}\text{ is not invertible on }X_s\}.
\end{equation}

We now prove that $\mathcal{E}$ does not meet the upper half-plane $\{z\in \mathbb{C}:\operatorname{Im}z>0\}$. It remains to show the injectivity of $\mathcal B_{0,k_s}$ when $\operatorname{Im}k_s>0$. Let $(\bm\varphi,\bm c)\in\ker\mathcal B_{0,k_s}$, and set
\begin{equation}
    \pi=Q^{0,k_s}[\bm\varphi], \qquad \mathbf w=\mathbf S^{0,k_s}[\bm\varphi].
\end{equation}
Define the exterior fields
\[
    q=\pi+\sum_{j=1}^N c_j q_{j,0}^* , \qquad \mathbf v=\mathbf w+\sum_{j=1}^N c_j\mathbf v^*_{j,0,k_s}
    \qquad \mathrm{in} \ \mathbb{R}^3\setminus \overline{D}.
\]
By \eqref{jumpform} and the first component of $\mathcal B_{0,k_s} (\bm\varphi,\bm c)=0$, the pair $(q,\mathbf{v})$ satisfies 
\begin{equation}\left\{
\begin{aligned}
    &\mathscr{L}_{k_s}(q,\mathbf{v})=0, \quad \nabla \cdot \mathbf{v}=0  \qquad \mathrm{in}\ \mathbb{R}^3 \setminus \overline{D} , \\
    &\mathbf{t}_{\mathbf{n}}(q,\mathbf{v})=0.
\end{aligned}\right.
\end{equation}
By Remark \ref{outgoingforhydrolayerpotential}, the single-layer potentials and monopole fields satisfy the outgoing radiation condition \eqref{radiationcond}. Hence, Lemma~\ref{uniquekpzero} implies that $(q,\mathbf v)=(0,0)$.

Since $\mathbf w$ is continuous across $\partial D$ and is divergence free in $D$, 
the flux $\int_{\partial D_i} \mathbf{w}\cdot \mathbf{n} \,dS=0$ for each $i$. On the other hand, $\Delta G^0=\delta_0$ gives
\[
    \int_{\partial D_i}\mathbf v^*_{j,0,k_s}\cdot\mathbf n\,dS =\frac{\delta_{ij}}{k_s^2}.
\]
Taking the flux of $\mathbf v$ through each $\partial D_i$ then gives $\bm c=0$, and consequently,
\begin{equation}
    \pi=0 , \quad \mathbf{w}=0 \qquad \mathrm{in} \ \mathbb{R}^3\setminus \overline{D}.
\end{equation}
Testing $\mathcal{L}_{k_s}(\pi,\mathbf{w})=0$ with $\overline{\mathbf w}$ in each $D_i$ and using $\nabla \cdot \mathbf{w}=0$ give
\[
    2\int_{D_i}|\nabla^{\mathrm s}\mathbf w|^2\,dx =k_s^2\int_{D_i}|\mathbf w|^2\,dx.
\]
As $k_s^2\notin[0,\infty)$, we conclude that
\begin{equation}
    \pi=\sum_{i=1}^NC_i\mathbbm{1}_{D_i} , \quad \mathbf{w}=0 \qquad \mathrm{in} \ \mathbb{R}^3,
\end{equation}
for some constants $C_1,\cdots,C_N$. The jump formula in \eqref{jumpform} gives $\bm\varphi=C_i\mathbf n$ on $\partial D_i$. Using the second component of $\mathcal B_{0,k_s} (\boldsymbol\varphi,\bm c)=0$, we obtain 
\begin{equation}
    0= \int_{\partial D_i} \bm\varphi\cdot \mathbf{n} \,dS = C_i |\partial D_i|.
\end{equation}
It follows that $\bm\varphi=0$. Thus, $\mathcal B_{0,k_s}$ is injective when $\operatorname{Im}k_s>0$. This proves (b).

Finally, fix $s\in \mathbb{R}$ and a compact set $\mathcal C\subset\mathbb C^*\setminus\mathcal E$. The set
\begin{equation}
    \mathcal{O}_s := \{(k_p,k_s)\in \mathcal{U}: \mathcal B_{k_p,k_s} \text{ is invertible on }X_s \}
\end{equation}
is open and $\mathcal B_{k_p,k_s}^{-1}$ is holomorphic on $\mathcal{O}_s$. The preceding arguments give $\{0\}\times\mathcal C\subset\mathcal O_s$. So, the compactness of $\mathcal{C}$ gives $\varepsilon=\varepsilon(\mathcal{C},s,D)>0$ such that $\{z\in \mathbb{C}:|z|\leq\varepsilon\}\times\mathcal C\subset\mathcal O_s$. Moreover, for $|k_p|<\varepsilon$, 
\begin{equation}
    \mathcal B_{k_p,k_s}^{-1}-\mathcal B_{0,k_s}^{-1}
    =k_p\int_0^1\partial_{k_p}\mathcal B_{tk_p,k_s}^{-1}\,dt=\mathcal{O}_{\mathcal L(X_s)}(|k_p|),
\end{equation}
uniformly for $k_s\in \mathcal C$. This proves (c).
\end{proof}

\begin{lemma}\label{lemma35vkpks}
Let $s\in \mathbb{R}$, $\mathcal{C}\subset \mathbb{C}^*\setminus \mathcal{E}$ be compact, and $\mathcal{V}= \mathcal{V}(s,\mathcal{C})$ be in \eqref{defofVeps}. For every $(k_p,k_s)\in \mathcal{V}$ and $g\in H^s(\partial D)$, problem \eqref{uindefDtN} admits a unique distributional solution $(q,\mathbf v,\mathbf f)$ in the sense of Remark~\ref{rem:generalized-solutions}.
The solution can be represented by
\begin{equation}\label{exteriorLayerRepresentation}
    q=Q^{k_p,k_s}[\bm\varphi]+\sum_{j=1}^{N}c_j q_{j,k_p}^*, \quad \mathbf v=\mathbf S^{k_p,k_s}[\bm\varphi]+\sum_{j=1}^{N}c_j\mathbf v^*_{j,k_p,k_s},
    \quad 
\end{equation}
where
\begin{equation}\label{definevarphic}
    (\bm\varphi,\bm c)=\mathcal B_{k_p,k_s}^{-1}(g\mathbf n,\bm0) \in H^s(\partial D;\mathbb{C}^3) \times \mathbb{C}^N.
\end{equation}
The generalized exterior velocity trace $\mathbf f$ is given by
\begin{equation}\label{exteriorVelocityTrace}
    \mathbf f =\mathbf S^{k_p,k_s}[\bm\varphi]+\sum_{j=1}^{N}c_j\mathbf v^*_{j,k_p,k_s}\big|_{\partial D}
    \in H^{s+1}(\partial D;\mathbb{C}^3).
\end{equation}
\end{lemma}

\begin{proof}
For brevity, write $H^s=H^s(\partial D;\mathbb C^3)$.

\emph{Existence.} By Lemma \ref{lemma34Bkpks}, the inverse $\mathcal B_{k_p,k_s}^{-1}$ exists for $(k_p,k_s)\in \mathcal{V}$, so $(\bm\varphi,\bm c)$ in \eqref{definevarphic} is well-defined. The mapping properties of the hydrodynamic layer potentials in Lemma \ref{lemmaKmaping}(a) show that \eqref{exteriorLayerRepresentation}--\eqref{exteriorVelocityTrace} define a distributional triple $(q,\mathbf v,\mathbf f)$ with the stated regularity. By \eqref{jumpform}, the pair $(q,\mathbf v)$ satisfies the distributional equation
\begin{equation}
    \mathscr L_{k_s}(q,\mathbf v)=0, \quad \mathscr D_{k_p,k_s}(q,\mathbf v)=0 \qquad \mathrm{in} \ 
    \mathcal{D}'(\mathbb R^3\setminus\overline D).
\end{equation}
By Remark~\ref{outgoingforhydrolayerpotential}, the hydrodynamic single-layer potentials and the monopole fields satisfy the outgoing radiation condition \eqref{radiationcond}. Thus, \eqref{distributionsolutionqv} holds.

It remains to verify the generalized boundary condition \eqref{distributionboundarycondition}. For smooth densities, \eqref{distributionboundarycondition} follows from Green's formula and the jump formula in \eqref{jumpform}.
For general $\boldsymbol\varphi\in H^s$, choose smooth densities converging to $\boldsymbol\varphi$ in $H^s$. The associated layer potentials converge smoothly on compact subsets of the exterior domain, while their velocity traces and tractions converge in $H^{s+1}$ and $H^s$, respectively, by Lemma~\ref{lemmaKmaping}(a) and the jump formula. Passing to the limit in Green's identity yields \eqref{distributionboundarycondition}. Hence, the triple $(q,\mathbf v,\mathbf f)$ defined by \eqref{exteriorLayerRepresentation}--\eqref{exteriorVelocityTrace} is a distributional solution of \eqref{uindefDtN}.

\emph{Uniqueness.} First, we recall a reciprocity identity for outgoing solutions. Define
\[
 W_R(a,d):=\int_{\partial B_R}(\partial_r a\cdot d-a\cdot\partial_r d)\,dS.
\]
Let $(p,\mathbf u)$ and $(\pi,\mathbf w)$ be two solutions of \eqref{distributionsolutionqv} with $D$ replaced by $B_{R_0}$. By the outgoing radiation conditions, for every $R>R_0$, a direct computation verifies 
\begin{equation}\label{rellich}
    \begin{aligned}
 &\int_{\partial B_R}\bigl[\mathbf t_{\mathbf n_R}(p,\mathbf u)\cdot\mathbf w
   -\mathbf u\cdot\mathbf t_{\mathbf n_R}(\pi,\mathbf w)\bigr]dS\\
 &\qquad=W_R(\mathbf u_s,\mathbf w_s)+\frac{k_s^2}{(k_s^2-2k_p^2)^2}W_R(p,\pi)=0.
\end{aligned}
\end{equation}

Now, let $(p,\mathbf u,\mathbf h)$ be a distributional solution of \eqref{uindefDtN} with $g=0$. To prove that this solution vanishes, we construct test fields with arbitrary compactly supported forcing and boundary traction.
Fix $\mathbf H\in C_0^\infty(\mathbb{R}^3\setminus \overline{D};\mathbb C^3)$ and $\bm\eta\in C^\infty(\partial D;\mathbb C^3)$. Extend $\mathbf H$ by zero to $\mathbb R^3$ and define the volume potentials
\[
    \pi_0:=\sum_{i=1}^3\Pi_i^{k_p,k_s}*H_i,\qquad \mathbf w_0:=\mathbf\Gamma^{k_p,k_s}*\mathbf H,
\]
where $(\Pi^{k_p,k_s}, \boldsymbol\Gamma^{k_p,k_s})$ are the hydrodynamic fundamental solutions defined in Appendix \ref{appenlayerpotential}. Thus,
\begin{equation}\label{volumepotential}
    \mathscr L_{k_s}(\pi_0,\mathbf w_0)=\mathbf H, \quad \mathscr D_{k_p,k_s}(\pi_0,\mathbf w_0)=0 \qquad \mathrm{in} \ \mathbb{R}^3.
\end{equation}
Moreover, by \eqref{deps}--\eqref{depsd}, $\mathbf{w}_0$ satisfies the outgoing radiation condition \eqref{radiationcond}.
Define
\begin{equation}\label{defofpsidpiw}
\left\{
    \begin{aligned}
        (\bm\psi,\bm d) &:=\mathcal B_{k_p,k_s}^{-1}  \bigl(\bm\eta-\mathbf t_{\mathbf n}(\pi_0,\mathbf w_0),
          \bm0\bigr), \\
          \pi&:=\pi_0+Q^{k_p,k_s}[\bm\psi]+\sum_{j=1}^N d_j q^*_{j,k_p},\\
    \mathbf w&:=\mathbf w_0+\mathbf S^{k_p,k_s}[\bm\psi]+\sum_{j=1}^N d_j\mathbf v^*_{j,k_p,k_s}.
    \end{aligned}\right.
\end{equation}
By \eqref{volumepotential}, \eqref{jumpform}, and Remark~\ref{outgoingforhydrolayerpotential},
\begin{equation}\label{equationforpiw}
    \left\{
    \begin{aligned}
        & \mathscr L_{k_s}(\pi,\mathbf w)=\mathbf H,\qquad
    \mathscr D_{k_p,k_s}(\pi,\mathbf w)=0\qquad\text{in }\mathbb{R}^3\setminus \overline{D}, \\
    & \mathbf t_{\mathbf n}(\pi,\mathbf w)=\bm \eta, \\
    & \mathbf{w} \textrm{ satisfies the outgoing radiation condition }\eqref{radiationcond}.
    \end{aligned}
    \right.
\end{equation}

Next, we verify the smoothness of $(\pi,\mathbf w)$ as a test field. Since $\mathbf H$ is smooth, the volume potential $(\pi_0,\mathbf{w}_0)$ is smooth. Then, the first line of \eqref{defofpsidpiw} gives
\begin{equation}
    \big(\frac{1}{2}\mathbf I+\mathbf K^{k_p,k_s,*}\big)\bm\psi\in C^\infty(\partial D;\mathbb C^3).
\end{equation}
By the principal symbol of $\mathbf K^{k_p,k_s,*}$ in Lemma \ref{lemmaKmaping}(d), the operator
$\frac12 \mathbf{I}+\mathbf K^{k_p,k_s,*}$ is elliptic for $(k_p,k_s)\in \mathcal{V}$.
Elliptic regularity yields $\bm\psi\in C^\infty(\partial D;\mathbb C^3)$ \cite[§18.1]{HormanderIII}. 
The Dirichlet boundary regularity then
implies that $(\pi,\mathbf w)$ is smooth up to $\partial D$ from the exterior \cite[§6.3.2]{evans2022partial}.

Extend $(\pi,\mathbf w)$ smoothly across $\partial D$. 
Fix $R>0$ such that $\overline D\cup\operatorname{supp}\mathbf H\subset B_R$,
and choose a cutoff function $\chi\in C_0^\infty(B_{2R})$ with $\chi= 1$ on $B_R$.
Since $\mathbf H$ vanishes near $\partial D$, $(\chi\pi,\chi\mathbf w)$ is an admissible test pair in \eqref{distributionboundarycondition}. Applying \eqref{distributionboundarycondition} with $g=0$ and the test pair $(\chi \pi, \chi \mathbf{w}$), and using \eqref{equationforpiw}, we get
\[
\begin{aligned}
 -\langle\mathbf h,\bm\eta\rangle_{\partial D}
 &=\langle\mathbf u,\mathbf H\rangle_{\mathbb{R}^3\setminus \overline{D}}
   +\int_{B_{2R}\setminus B_R} \bigl[\mathbf u\cdot\mathscr L_{k_s}(\chi\pi,\chi\mathbf w)
          +p\mathscr D_{k_p,k_s}(\chi\pi,\chi\mathbf w)\bigr]\,dx\\
 &=\langle\mathbf u,\mathbf H\rangle_{\mathbb{R}^3\setminus \overline{D}}
   + \int_{\partial B_R}\left[\mathbf t_{\mathbf n_R}(p,\mathbf u)\cdot\mathbf w
        -\mathbf u\cdot\mathbf t_{\mathbf n_R}(\pi,\mathbf w)\right]\,dS \\
 &=\langle\mathbf u,\mathbf H\rangle_{\mathbb{R}^3\setminus \overline{D}} ,
\end{aligned}
\]
where the last equality uses the reciprocity identity \eqref{rellich}. Since $\mathbf H$ and $\boldsymbol\eta$ are arbitrary, taking $\boldsymbol\eta=0$ gives
$\mathbf u=0$, while taking $\mathbf H=0$ gives $\mathbf h=0$. The equation
$\mathscr L_{k_s}(p,\mathbf u)=0$
then implies $\nabla p=0$.
As $\mathbb{R}^3\setminus\overline{D}$ is connected and $p$ is outgoing at $k_p$, $p=0$. This proves uniqueness.
\end{proof}

\begin{corollary}\label{coroNkpks}
 Let $s\in \mathbb{R}$, $\mathcal{C} \subset \mathbb{C}^*\setminus \mathcal{E}$ be compact, 
 and $\mathcal{V}=\mathcal{V}(s,\mathcal{C})$ be in \eqref{defofVeps}. For every $(k_p,k_s) \in  \mathcal{V}$,
the operator $\mathcal{N}_{k_p,k_s}$ maps $H^{s}(\partial D)$ into $H^{s+1}(\partial D)$, and satisfies
    \begin{equation}\label{nhntdasymptotic}
        \mathcal{N}_{k_p,k_s} = \mathcal{N}_{0,k_s} + \mathcal{O}_{\mathcal{L}(H^{s}(\partial D), H^{s+1}(\partial D) )}(|k_p|), \qquad \mathrm{as} \ k_p \rightarrow 0,
    \end{equation}
    uniformly for $k_s\in\mathcal C$. 
    
    Moreover, $\mathcal N_{k_p,k_s}$ is an elliptic pseudodifferential operator of order $-1$ on $\partial D$,
with principal symbol
\begin{equation}\label{principalSymbolNkpks}
    \sigma_{-1}(\mathcal N_{k_p,k_s})(x,\xi)=
    -\frac{1}{2(1-k_p^2/k_s^2)|\xi|_g}, \qquad (x,\xi)\in T^*\partial D\setminus0,
\end{equation}
where $g$ is the Riemannian metric on $\partial D$ induced by the embedding $\iota:\partial D\rightarrow \mathbb{R}^3$.
\end{corollary}

\begin{proof}
By Lemma \ref{lemma35vkpks}, for $(k_p,k_s) \in  \mathcal{V}$, the nhNtD operator admits the representation
\[
    \mathcal N_{k_p,k_s}[g]= \mathbf n\cdot \mathbf S^{k_p,k_s}[\boldsymbol\varphi]
    +\sum_{j=1}^N c_j\mathbf n\cdot \mathbf v_{j,k_p,k_s}^* \big|_{\partial D},
\]
where $(\boldsymbol\varphi,\bm c) =\mathcal B_{k_p,k_s}^{-1} (g\mathbf n,\mathbf 0)$. The mapping properties and holomorphy of $\mathbf{S}^{k_p,k_s}$ and $\mathcal{B}_{k_p,k_s}^{-1}$ in Lemma \ref{lemmaKmaping}(a) and 
Lemma \ref{lemma34Bkpks}, respectively, imply that $\mathcal{N}_{k_p,k_s}$ is holomorphic near $\mathcal{V}$ with values in 
$\mathcal{L}(H^{s}(\partial D),H^{s+1}(\partial D))$. A Taylor expansion in $k_p$ at zero then yields \eqref{nhntdasymptotic}, uniformly for $k_s\in \mathcal{C}$.

Next, we compute the principal symbol of $\mathcal{N}_{k_p,k_s}$. Denote $\vartheta=k_p^2/k_s^2$. By the principal symbol of $\mathbf K^{k_p,k_s,*}$ in Lemma \ref{lemmaKmaping}(d), the operator
$\frac12 \mathbf{I}+\mathbf K^{k_p,k_s,*}$ is elliptic for $(k_p,k_s)\in \mathcal{V}$.
Let $\mathbf{P}$ be a parametrix of $\frac12 \mathbf{I}+\mathbf K^{k_p,k_s,*}$ so that
\begin{equation}
    \sigma_0(\mathbf{P}) = \Big( \frac{1}{2}\mathbf{I} +  \sigma_0( \mathbf K^{k_p,k_s,*})\Big)^{-1};
\end{equation}
see \cite[§18.1]{HormanderIII}. The map $\mathcal{F}_{k_p,k_s}$ and the functional $\mathcal{L}$ defined in \eqref{defofFandL} are smoothing. So, the first component of
$\mathcal B_{k_p,k_s}^{-1}(g\mathbf n,\mathbf 0)$ agrees with $\mathbf P[g\mathbf n]$
modulo a smoothing operator. Consequently, there is a smoothing operator $\mathcal{R}$ such that
\begin{equation}\label{formulaofNkpks}
    \mathcal N_{k_p,k_s}[h]=\mathbf n\cdot\mathbf S^{k_p,k_s}\bigl[\mathbf P[h\mathbf n]\bigr]
    +\mathcal R[h].
\end{equation}
Thus, $\mathcal N_{k_p,k_s}$ is a pseudodifferential operator of order at most $-1$. By Lemma \ref{lemmaKmaping}(d), using $e\cdot\mathbf n=0$ and $|e|=|\mathbf n|=1$, we obtain
\[
    \mathbf n^\top \sigma_{-1}(\mathbf S^{k_p,k_s})=-\frac{1+\vartheta}{4|\xi|_g}\,\mathbf n^\top, \qquad
    \mathbf n^\top \sigma_0(\mathbf P)\mathbf n =\frac{2}{1-\vartheta^2}.
\]
The composition formula for principal symbols and \eqref{formulaofNkpks} therefore gives
\begin{equation}
    \sigma_{-1}(\mathcal N_{k_p,k_s}) = \mathbf n^\top \sigma_{-1}(\mathbf S^{k_p,k_s})\sigma_0(\mathbf P)\mathbf n
    =-\frac{1}{2(1-\vartheta)|\xi|_g}.
\end{equation}
This proves \eqref{principalSymbolNkpks}. Since this symbol does not vanish for $\xi\neq0$ and $(k_p,k_s) \in  \mathcal{V}$, $\mathcal N_{k_p,k_s}$ is elliptic of order $-1$.
\end{proof}

\subsection{Asymptotic analysis of Minnaert-type resonances}

We begin by defining the resonances of the transmission problem \eqref{nondfinal1} in the distributional framework introduced above.

\begin{definition} \label{def:resonant-wavenumber}
Fix $\delta,\tau,\gamma,k_p,k_s$ for which the transmission problem \eqref{nondfinal1} and the outgoing radiation condition \eqref{radiationcond} are well-defined. A number $k\in\mathbb C^*$ is called a \emph{resonance} if \eqref{nondfinal1} admits a nontrivial solution $(p,q,\mathbf v)$ in the following sense.
\begin{itemize}
    \item[(a)] The function $p\in H^1(D)$ satisfies $(\Delta+k^2 \tau^2 )p=0$ in $\mathcal D'(D)$ with well-defined boundary data
\[
    p|_-\in H^{1/2}(\partial D)\quad \mathrm{and} \quad \frac{\partial p}{\partial \mathbf{n}}\Big|_-\in H^{-1/2}(\partial D).
\]
\item[(b)] There exists $\mathbf f\in\mathcal D'(\partial D;\mathbb C^3)$ such that $(q,\mathbf v,\mathbf f)$ is a distributional solution of the exterior problem \eqref{uindefDtN} in the sense of Remark~\ref{rem:generalized-solutions}, where
\begin{equation}\label{formulagncdotf}
    g=k_s^2 \left[ \gamma \mathcal{H}\left(\frac{1}{k^2}\left. \frac{\partial p}{\partial \mathbf{n}}\right|_- \right) -\frac{\delta}{k^2}  p\right],
    \qquad
    \mathbf n\cdot\mathbf f=\frac{1}{k^2}\frac{\partial p}{\partial \mathbf{n}}\Big|_-
    \quad\text{in }\mathcal D'(\partial D).
\end{equation}
\end{itemize}
The field $(p,q,\mathbf v)$ is called a \emph{resonant state}.
\end{definition}

Define the boundary operator
\begin{equation}\label{defofFinA}
 F (k,\delta;k_p,k_s):= (I-k_s^2 \gamma \mathcal{N}_{k_p,k_s} \mathcal{H}) \big(-\frac{1}{2}I + K^{k\tau,*} \big) + \delta k_s^2 \mathcal{N}_{k_p,k_s} S^{k\tau}.
\end{equation}

\begin{proposition}\label{prop:boundary-reduction}
Let $\gamma \in \mathbb{C}^*$ and $\mathcal{C} \subset \mathbb{C}^*\setminus \mathcal{E}$ be compact. For 
$|k_p|<\varepsilon$ and $k_s \in \mathcal{C}$, where $\varepsilon=\varepsilon(\mathcal{C},D)$ is chosen small enough so that Lemma~\ref{lemma35vkpks} holds for both $s=-5/2$ and $s=-2$. Then,
the following hold.

\begin{itemize}
    \item[(i)] $F(k,\delta;k_p,k_s)$ is an elliptic pseudodifferential operator of order $1$.
    \item[(ii)] For every $k\in \mathbb{C}^*$ with $|k\tau|$ sufficiently small, $k$ is a resonance of \eqref{nondfinal1} if and only if 
    \begin{equation}\label{operatorFL2Hminusone}
        F(k,\delta;k_p,k_s) : L^2(\partial D) \rightarrow H^{-1}(\partial D)
    \end{equation}
    is not invertible. Moreover, \eqref{operatorFL2Hminusone} is Fredholm of index zero.
\end{itemize}
\end{proposition}

\begin{proof}
We first prove (i). By Corollary~\ref{coroNkpks}, $\mathcal{N}_{k_p,k_s}$ has order $-1$. By the proof of Lemma~\ref{lemmaC2}, $S^{k\tau}$ and $K^{k\tau,*}$ have order $-1$. Since $\mathcal{H}$ has order $2$, $F$ has order $1$ and the principal symbol
\begin{equation}\label{symbolofF}
    \sigma_1(F)=\frac{k_s^2\gamma}{2} \sigma_{-1} ( \mathcal{N}_{k_p,k_s} ) \sigma_2(\mathcal{H})= -\frac{k_s^2\gamma}{4(1-k_p^2/k_s^2)}|\xi|_g,
    \qquad (x,\xi)\in T^*\partial D\setminus0,
\end{equation}
where we used \eqref{principalSymbolNkpks} and $\sigma_{2}(\mathcal{H})=|\xi|_g^2$.
Since $\gamma\neq0$ and $|k_p|<|k_s|/\sqrt2$, $\sigma_1(F)$ does not vanish for $\xi \neq 0$. Thus, $F$ is elliptic.

We next prove (ii). By \eqref{symbolofF}, we may find a bounded operator $R:L^2(\partial D)\rightarrow L^2(\partial D)$ such that
\[
    F=- \frac{k_s^2\gamma}{4(1-k_p^2/k_s^2)} (I-\Delta_{\partial D} )^{1/2} +R.
\]
Since $(I-\Delta_{\partial D} )^{1/2}$ is an isomorphism between $L^2(\partial D)$ and $H^{-1}(\partial D)$, and $R:L^2(\partial D )\to H^{-1}(\partial D )$ is compact, \eqref{operatorFL2Hminusone} is Fredholm of index zero. 

Assume that $k$ is a resonance of \eqref{nondfinal1} with a resonant state $(p,q,\mathbf v)$.
    Since the operator $S^0:H^{-1/2}(\partial D)\rightarrow H^{1/2}(\partial D)$ is invertible \cite{ammari2018mathematical}, the same holds for $S^{k\tau}$ when $|k\tau|$ is sufficiently small by Lemma \ref{lemmaC2}. We may therefore define
    \begin{equation}
        \varphi := (S^{k\tau})^{-1} [p|_-] \in H^{-1/2}(\partial D).
    \end{equation}
   Taking $|k\tau|$ sufficiently small also ensures that $(k\tau)^2$ is not a Dirichlet eigenvalue of
$-\Delta$ in $D$. We deduce that $p=S^{k\tau}[\varphi]$ in $D$. The jump formula gives
    \begin{equation}\label{psiequaKvarphi}
        \frac{\partial p}{\partial \mathbf{n}}\Big|_-= \big( -\frac{1}{2}I+K^{k\tau,*} \big) \varphi \in H^{-1/2}(\partial D).
    \end{equation}
    By \eqref{formulagncdotf} and the mapping property $\Delta_{\partial D}:H^{-1/2}(\partial D)\rightarrow H^{-5/2}(\partial D)$, we have 
    \begin{equation}\label{formulaofg52}
        g=\frac{k_s^2}{k^2} \left\{ \gamma \mathcal{H} \big(-\frac{1}{2}I + K^{k\tau,*} \big) -\delta S^{k\tau} \right\} \varphi
    \in H^{-5/2}(\partial D).
    \end{equation}
    The generalized problem \eqref{distributionsolutionqv}--\eqref{distributionboundarycondition} is uniquely solvable for above $g$, proved in Lemma~\ref{lemma35vkpks} with $s=-5/2$. Thus, with \eqref{formulagncdotf} and \eqref{formulaofg52}, we get
\begin{equation}\label{psiequalks2phi}
    \frac{\partial p}{\partial \mathbf{n}}\Big|_-=k^2\mathbf n\cdot\mathbf f =k^2 \mathcal{N}_{k_p,k_s}[g]
    = k_s^2\mathcal{N}_{k_p,k_s} \left\{ \gamma \mathcal{H} \big(-\frac{1}{2}I + K^{k\tau,*} \big) -\delta S^{k\tau} \right\}\varphi .
\end{equation}
Substituting \eqref{psiequaKvarphi} into \eqref{psiequalks2phi} yields
$F\varphi=0$. Since $F$ is elliptic, elliptic regularity implies that $\varphi \in L^2(\partial D)$.
Moreover, $\varphi$ is nontrivial. Indeed, if $\varphi=0$, then $p=0$ and $g=0$. The uniqueness result of Lemma~\ref{lemma35vkpks} would imply $(q,\mathbf v)= 0$, contradicting the nontriviality of the resonant state.

Conversely, suppose that \eqref{operatorFL2Hminusone} is not invertible. Since it is Fredholm of index zero, we may find a nonzero $\varphi \in L^2(\partial D)$ such that $F\varphi =0$. Define $p:=S^{k\tau}[\varphi]\in H^1(D)$. Then 
\begin{equation}
    (\Delta+k^2\tau^2)p=0 \qquad \mathrm{in} \ D.
\end{equation}
The jump formula gives
\begin{equation}\label{partialppartialn}
    \frac{\partial p}{\partial \mathbf{n}}\Big|_- = \big(-\frac{1}{2}I + K^{k\tau,*} \big)\varphi.
\end{equation}
Define
\begin{equation}\label{definegHminus2}
    g:=\frac{k_s^2}{k^2}
       \Big(\gamma\mathcal H\big(-\frac{1}{2}I + K^{k\tau,*} \big) -\delta S^{k\tau}\Big) \varphi
       \in H^{-2}(\partial D).
\end{equation}
By Lemma~\ref{lemma35vkpks}, applied with $s=-2$, the exterior problem \eqref{uindefDtN} with boundary traction $g\mathbf n$ admits a unique solution
$(q,\mathbf v,\mathbf f)$ such that $\mathbf n\cdot\mathbf f=\mathcal{N}_{k_p,k_s}[g]$.
The equations $F\varphi=0$, \eqref{partialppartialn}, and \eqref{definegHminus2} give
\[
    k^2\mathbf n\cdot\mathbf f=k^2\mathcal{N}_{k_p,k_s}[g]=\frac{\partial p}{\partial \mathbf{n}}\Big|_- \quad\text{in }H^{-1}(\partial D).
\]
Hence, both interface conditions in \eqref{formulagncdotf} hold. Finally, $\varphi\neq0$ and the invertibility of $S^{k\tau}$ imply that $p|_-\neq0$, and hence $(p,q,\mathbf{v})$ is a resonant state, completing the proof.
\end{proof}

For brevity, we write $S=S^0$ and $K^*=K^{0,*}$. For $j=1,\cdots,N$, let
\begin{equation}\label{componentStaticModes}
    \varphi_j:=S^{-1}[\mathbbm{1}_{\partial D_j}],
    \qquad \ell_j(h):=\int_{\partial D_j}h\,dS.
\end{equation}
Then,
\begin{equation}\label{componentKernelRange}
\ker(-\frac12 I+K^*)=\operatorname{span}\{\varphi_1,\cdots,\varphi_N\},\quad 
    \operatorname{ran}(-\frac12 I+K^*)=\bigcap_{j=1}^N\ker\ell_j.
\end{equation}
See \cite[Lemma 2.1]{ammari2018minnaert}. It follows that $\operatorname{span}\{\varphi_1,\cdots,\varphi_N\} $ is contained in the kernel of the limiting operator
\begin{equation}
    F(0,0;0,k_s)= (I-k_s^2 \gamma \mathcal{N}_{0,k_s} \mathcal{H}) \big(-\frac{1}{2}I + K^{*} \big) .
\end{equation}
Unlike the classical acoustic setting, however, $\ker F(0,0;0,k_s)$ need not be $N$ dimensional. Additional kernel elements arise precisely when
\begin{equation}
    \ker(I-k_s^2\gamma \mathcal N_{0,k_s}\mathcal H )
    \cap \operatorname{ran}(-\frac{1}{2}I+K^* ) \neq \{ 0\}.
\end{equation}
These elements, corresponding to nontrivial ``capillary-viscous surface modes'', are absent from the classical acoustic model without surface tension. 

To exclude these modes that are irrelevant to Minnaert-type resonances, we impose the following sufficient \emph{non-resonance condition}:
\begin{equation}\label{assumption:nonresonance-A0}
    I-k_s^2\gamma \mathcal N_{0,k_s}\mathcal H:L^2(\partial D)\rightarrow H^{-1}(\partial D) \textrm{ is invertible.}
\end{equation}
The condition also has a natural interpretation in terms of the exterior problem \eqref{uindefDtN}: it excludes nontrivial solutions of \eqref{uindefDtN} at $k_p=0$ satisfying the capillary-viscous boundary condition
\[
\mathbf t_{\mathbf n}(q,\mathbf v) =k_s^2\gamma
\mathcal H(\mathbf v\cdot\mathbf n)\mathbf n \qquad\text{on }\partial D.
\]

The following result shows that the non-resonance condition \eqref{assumption:nonresonance-A0} holds ``generically''.

\begin{lemma}\label{lemma39bijection}
    For every $\gamma\in\mathbb C^*$, the family
    \begin{equation}\label{familyIminusks2}
        k_s\mapsto I - k_s^2 \gamma \mathcal{N}_{0,k_s}\mathcal{H}  \in\mathcal L(L^2(\partial D),H^{-1}(\partial D))
    \end{equation}
is holomorphic on $\mathbb C^*\setminus\mathcal E$ and consists of Fredholm operators of index zero.
The set
\[
    E_\gamma:=\{k_s\in \mathbb C^*\setminus\mathcal E: I-k_s^2\gamma \mathcal N_{0,k_s}\mathcal H \textrm{ is not invertible}\}
\]
is closed and discrete in $\mathbb C^*\setminus\mathcal E$. 
\end{lemma}

\begin{proof}
By the principal symbol
\begin{equation}
    \sigma_1(I - k_s^2 \gamma \mathcal{N}_{0,k_s}\mathcal{H}) = \frac{k_s^2\gamma}{2}|\xi|_g,
\end{equation}
and the same argument as in the proof for Corollary \ref{coroNkpks} and for Proposition \ref{prop:boundary-reduction}, the map \eqref{familyIminusks2} is holomorphic on $\mathbb C^*\setminus\mathcal E$, and its values there are Fredholm operators of index zero. 

To apply the analytic Fredholm theorem, it remains to find one $k_s$ at which $I - k_s^2 \gamma \mathcal{N}_{0,k_s}\mathcal{H}  $ is invertible. Choose $w\in\{\mathrm i/\gamma,-\mathrm i/\gamma\}$ such that
$\operatorname{Re}w\le0$, and let $\zeta$ be the square root of $w$ with $\operatorname{Im}\zeta>0$.
Such a root exists because $w\notin[0,\infty)$. In particular,
\[
    \zeta\in \mathbb C^*\setminus\mathcal E,\qquad \operatorname{Re}(\zeta^2)\le0,\qquad
    \zeta^2\gamma\in\{\mathrm i,-\mathrm i\}.
\]

Suppose that $(I - \zeta^2 \gamma \mathcal{N}_{0,\zeta}\mathcal{H} )\phi=0$.
Elliptic regularity implies $\phi\in C^\infty(\partial D)$.
By Lemma~\ref{lemma35vkpks}, let $(q,\mathbf v,\mathbf f)$ be the solution of \eqref{uindefDtN} at
$(k_p,k_s)=(0,\zeta)$ with $g=\mathcal H\phi$. Since $g$ is smooth, the first row of \eqref{def:b} and \eqref{C20bounded} imply, by iteration, that the
density in \eqref{definevarphic} is smooth. The layer potential representation and Dirichlet boundary
regularity then give smoothness of $(q,\mathbf v)$ up to $\partial D$ from the exterior, as in the proof of
Lemma~\ref{lemma35vkpks}. In particular, $\mathbf f$ coincides with the classical boundary trace $\mathbf{v}|_+$.
We therefore have
\begin{equation}\label{uindefDtN1incompres12}
\left\{
\begin{aligned}
    &\mathscr L_{\zeta}(q,\mathbf v)=0,\qquad \nabla\cdot\mathbf v=0 &&\text{in }\mathbb{R}^3\setminus \overline{D},\\
    &\mathbf t_{\mathbf n}(q,\mathbf v)|_+ =(\mathcal H\phi)\mathbf n,\qquad
      \mathbf v\cdot\mathbf n|_+ =\frac{\phi}{\zeta^2\gamma} &&\text{on }\partial D
\end{aligned}
\right.
\end{equation}
in the classical sense.

As in the proof of Lemma~\ref{uniquekpzero}, multiplying the equation $\mathscr L_{\zeta}(q,\mathbf v)=0$ by $\overline{\mathbf v}$, integrating over $B_R\setminus\overline D$, and letting $R\to\infty$, we obtain
\begin{equation}\label{identitylimit2}
    \overline{\zeta^2\gamma}\int_{\mathbb{R}^3\setminus D} \bigl(2|\nabla^{\mathrm s}\mathbf v|^2 -\zeta^2|\mathbf v|^2\bigr)\,dx
    = - \int_{\partial D} (\mathcal H\phi)\overline\phi\,dS \in \mathbb{R}.
\end{equation}
Since $\zeta^2\gamma$ is purely imaginary, taking imaginary parts in \eqref{identitylimit2} gives
\[
    2\|\nabla^{\mathrm s}\mathbf v\|_{L^2(\mathbb{R}^3\setminus D)}^2 -\operatorname{Re}(\zeta^2)\|\mathbf v\|_{L^2(\mathbb{R}^3\setminus D)}^2=0.
\]
Both terms are nonnegative, so $\nabla^{\mathrm s}\mathbf v=0$. The outgoing radiation condition for $\mathbf{v}$ then implies that
$\mathbf v=0$. Equation \eqref{uindefDtN1incompres12} now gives $\phi=0$. Thus, $I - \zeta^2 \gamma \mathcal{N}_{0,\zeta}\mathcal{H} $ is injective and, being Fredholm of index zero, is invertible.

Finally, since $\mathbb C^*\setminus\mathcal E$ is connected and
$k_s\mapsto I - k_s^2 \gamma \mathcal{N}_{0,k_s}\mathcal{H}  $
is a holomorphic Fredholm family of index zero with an invertible point, the analytic Fredholm theorem implies that $E_\gamma$ is closed and discrete in $\mathbb C^*\setminus\mathcal E$. This completes the proof.
\end{proof}

\begin{definition}\label{def:admissible}
A pair $(\gamma,k_s)$ is called \emph{admissible} if
\[
    \gamma\in\mathbb C^* \quad \mathrm{and} \quad
    k_s\in \mathbb C^*\setminus(\mathcal E\cup E_\gamma),
\]
where $\mathcal E$ and $E_\gamma$ are defined in Lemmas~\ref{lemma34Bkpks} and~\ref{lemma39bijection},
respectively.
\end{definition}

We recall some results of the Gohberg--Sigal theory, which will be used to analyze the asymptotic behaviors of the resonances.
The proof is in Appendix \ref{appenxgstheory}.

\begin{proposition}\label{prop:gohberg-sigal}
Let $X$ and $Y$ be complex Banach spaces.
\begin{enumerate}
\item[\textup{(i)}]
Let $A(z):X\to Y$ be a holomorphic Fredholm family of index zero on a connected domain $\Omega\subset\mathbb C$, invertible at some point. A point $z_0$ is a \emph{characteristic value} if $A(z_0)$ is not
invertible. 
For every such point there are a Banach space $X_1$ and holomorphic isomorphisms
\[
    E(z):Y\to\mathbb C^r\oplus X_1, \qquad G(z):\mathbb C^r\oplus X_1\to X,
\]
such that, near $z_0$,
\begin{equation}\label{eq:GS-factorization}
    E(z)A(z)G(z)=\operatorname{diag}\bigl((z-z_0)^{m_1},\cdots,(z-z_0)^{m_r},I\bigr).
\end{equation}
Here $r=\dim\ker A(z_0)$ and the positive integers $m_j$ are uniquely determined up to permutation. The \emph{geometric} and \emph{algebraic multiplicities} are defined by
\[
    m_{\mathrm{geo}}(z_0;A):=r, \qquad m_{\mathrm{alg}}(z_0;A):=\sum_{j=1}^r m_j.
\]
A characteristic value is called \emph{simple} if its algebraic multiplicity is one.

\item[\textup{(ii)}]
Under the assumptions of \textup{(i)}, suppose that, near zero,
\[
    A(z)=A_0+z^2A_2+\mathcal O_{\mathcal L(X,Y)}(|z|^3),
\]
and that
\[
    \ker A_0=\operatorname{span}\{u_1,\cdots,u_r\}, \qquad
    \operatorname{ran}A_0=\bigcap_{i=1}^r\ker\ell_i,
\]
where the $u_j$ and the continuous functionals $\ell_i\in Y'$ are linearly independent.
If the matrix
\begin{equation}
    M=\big(\ell_i(A_2u_j) \big)_{i,j=1}^r
\end{equation}
is invertible, then all $m_j=2$ at $z_0=0$ in \eqref{eq:GS-factorization}. Moreover,
\begin{equation}\label{eq:GS-inverse}
    A(z)^{-1} =\frac1{z^2}\sum_{i,j=1}^r (M^{-1})_{ij}\,u_i\otimes\ell_j +\mathcal O_{\mathcal L(Y,X)}(|z|^{-1}),
\end{equation}
where $(u_i\otimes\ell_j)f :=\ell_j(f)u_i$.

\item[\textup{(iii)}]
Let $A,B:X\to Y$ be holomorphic Fredholm families of index zero near the closure of a smooth bounded domain $\Omega\subset\mathbb C$. If $A$ is invertible on $\partial\Omega$ and
\[
    \sup_{z\in\partial\Omega}\|A(z)^{-1}(B(z)-A(z))\|_{\mathcal L(X)}<1,
\]
then $A$ and $B$ have the same number of characteristic values in $\Omega$, counted with algebraic multiplicity.

\item[\textup{(iv)}]
For a $(z,\eta)$-jointly holomorphic Fredholm family $A(z,\eta)$ of index zero, a simple characteristic value $z_0$ at $\eta=\eta_0$ extends uniquely to a simple characteristic value $z(\eta)$ near $z_0$, holomorphic
for $\eta$ sufficiently close to $\eta_0$.

\item[\textup{(v)}]
Let $A$ be as in \textup{(i)} near the closure of a smooth bounded domain $\Omega$, with no characteristic values on $\partial\Omega$. For every scalar function $f$ holomorphic near $\overline\Omega$,
\begin{equation}\label{GSweightedResidue}
    \frac1{2\pi\mathrm i}\operatorname{tr}\oint_{\partial\Omega}f(z)A(z)^{-1}A'(z)\,dz
=\sum_{z_\nu\in\Omega}m_{\mathrm{alg}}(z_\nu;A)f(z_\nu).
\end{equation}
\end{enumerate}
\end{proposition}

For a nonzero resonance of \eqref{nondfinal1}, its algebraic multiplicity is understood as its algebraic multiplicity as a characteristic value of $k\mapsto F(k,\delta;k_p,k_s)$.

\begin{proposition}\label{prop:quasistatic-existence}
Fix $\tau>0$ and an admissible pair $(\gamma,k_s)$. There exist $\varepsilon,\delta_0,r>0$ such that,
for every $|k_p|<\varepsilon$ and $|\delta|<\delta_0$, the family $k\mapsto F(k,\delta;k_p,k_s)$ has exactly $2N$ resonances in $|k|<r$, counted with algebraic multiplicity. Every such value satisfies $|k|\leq C|\delta|^{1/2}$ uniformly for $|k_p|<\varepsilon$. Every nonzero characteristic value is a resonance of problem \eqref{nondfinal1}.  
\end{proposition}

\begin{proof}
Choose $r>0$ small enough so that Proposition~\ref{prop:boundary-reduction} applies for $|k|<r$. So, every resonance $k$ of problem \eqref{nondfinal1} is exactly the nonzero characteristic value of $k\mapsto F(k,\delta;k_p,k_s)$ when $|k|<r$.

We first consider the limiting case. At $\delta=0$,
\begin{equation}\label{familyFatdelta0}
    F(k,0;k_p,k_s)=\left(I-k_s^2\gamma\mathcal N_{k_p,k_s}\mathcal H\right)
    \left(-\frac12I+K^{k\tau,*}\right).
\end{equation}
By the admissibility of $(\gamma,k_s)$ and Corollary~\ref{coroNkpks},
\[
    I-k_s^2\gamma\mathcal N_{k_p,k_s}\mathcal H:L^2(\partial D)\to H^{-1}(\partial D)
\]
has a holomorphic, uniformly bounded inverse for $|k_p|\leq\varepsilon$ after decreasing $\varepsilon$.
Therefore, it  suffices to study $-\frac12I+K^{k\tau,*}$. This is a holomorphic Fredholm family of index zero by Lemma~\ref{lemmaC2}, and is invertible except for a discrete set; see \cite[§2.9.1]{ammari2018mathematical}.
By Lemma~\ref{lemmaC2},
\begin{equation}\label{halfKexpansion}
    -\frac12I+K^{k\tau,*}=\left(-\frac12I+K^*\right)+k^2\tau^2K_2^*
    +\mathcal O_{\mathcal L(L^2(\partial D))}(|k|^3).
\end{equation}
For every $i,j$, the jump formula and the divergence theorem give
\begin{equation}\label{integralhalfKexpansion}
    \ell_i\left[\left(-\frac12I+K^{k\tau,*}\right)\varphi_j\right]
    =-k^2\tau^2\int_{D_i}S^{k\tau}[\varphi_j]\,dx.
\end{equation}
Since $S[\varphi_j]=\delta_{ij}$ in $D_i$,
\begin{equation}\label{componentSecondOrderPairing}
    \ell_i(\tau^2K_2^*\varphi_j)=-\tau^2|D_i|\delta_{ij}.
\end{equation}
The resulting matrix is invertible. By \eqref{componentKernelRange} and
Proposition~\ref{prop:gohberg-sigal}(ii), $m_1=\cdots=m_N=2$, and
\begin{equation}\label{TkInversePrincipalPart}
    \left(-\frac12I+K^{k\tau,*}\right)^{-1}=-\frac1{\tau^2k^2}
    \sum_{j=1}^N\frac{\varphi_j\otimes\ell_j}{|D_j|}+\mathcal O_{\mathcal L(L^2(\partial D))}(|k|^{-1}).
\end{equation}
Therefore, the family $k\mapsto F(k,0;k_p,k_s)$ has exactly one characteristic value in $|k|<r$, namely $k=0$, with geometric multiplicity $N$ and algebraic multiplicity $2N$.

For subsequent use, define, for $0<|k|\leq r$,
\begin{equation}\label{collectiveRdefinition}
\begin{aligned}
    R(k,k_p)&:=F(k,0;k_p,k_s)^{-1}\partial_\delta F(k,0;k_p,k_s)\\
    &=k_s^2\left(-\frac12I+K^{k\tau,*}\right)^{-1}
    \left(I-k_s^2\gamma\mathcal N_{k_p,k_s}\mathcal H\right)^{-1}\mathcal N_{k_p,k_s}S^{k\tau},
\end{aligned}
\end{equation}
and
\begin{equation}\label{normalizedCollectiveF}
    \widehat F(k,\delta;k_p,k_s):=F(k,0;k_p,k_s)^{-1}F(k,\delta;k_p,k_s)=I+\delta R(k,k_p).
\end{equation}
The expansion \eqref{TkInversePrincipalPart} gives
\[
    \|\delta R(k,k_p)\|_{\mathcal L(L^2(\partial D))}\leq C|\delta||k|^{-2}
\]
uniformly for $|k_p|\leq\varepsilon$. For sufficiently small $\delta_0$, this norm is less than $1$ on
$|k|=r$. Proposition~\ref{prop:gohberg-sigal}(iii) gives the count $2N$. A Neumann series also gives invertibility whenever $0<|k|\leq r$ and $C|\delta|<|k|^2$. Thus, any characteristic value $k$ satisfies the bound $|k|<C|\delta|^{1/2}$. This completes the proof.
\end{proof}

We now calculate the leading-order collective resonances. For $\tau>0$ and admissible $(\gamma,k_s)$, define
$\mathfrak{C}=(\mathfrak{C}_{ij})_{i,j=1}^N$ by
\begin{equation}\label{collectiveCapacitanceMatrixkbzero}
\mathfrak{C}_{ij}:=\frac{k_s^2}{\tau^2|D_j|}\int_{\partial D_j}
    \left(I-k_s^2\gamma\mathcal N_{0,k_s}\mathcal H\right)^{-1}\mathcal N_{0,k_s}[\mathbbm{1}_{\partial D_i}]\,dS,
\end{equation}
where $\mathbbm{1}_{\partial D_i}$ is the characteristic function of $\partial D_i$ on $\partial D$. The matrix $\mathfrak C$ depends only on $D,\tau,\gamma,k_s$.

\begin{theorem}\label{thm314}\label{thm51}
Fix $\tau>0$ and an admissible pair $(\gamma,k_s)$. For each eigenvalue $\mu$ of $\mathfrak C$, let $m_\mu$ denote its
algebraic multiplicity and $p_\mu$ the size of its largest Jordan block.

Let $\varepsilon,\delta_0,r>0$ be in Proposition \ref{prop:quasistatic-existence}. Then,
for every $|k_p|<\varepsilon$ and $|\delta|<\delta_0$, the family
\[
    k\mapsto F(k,\delta;k_p,k_s)
\]
has exactly $2N$ characteristic values in $|k|<r$, counted with algebraic multiplicity. 
These characteristic values can be partitioned into groups
associated with the distinct eigenvalues of $\mathfrak C$ as follows.

\begin{enumerate}
\item[\textup{(a)}] 
For each nonzero eigenvalue $\mu$ of $\mathfrak C$, there are $2m_\mu$ nonzero resonances, counted with algebraic multiplicity, which can be labelled
\[
    k_{\mu,\ell}^{\pm}(\delta,k_p), \qquad \ell=1,\cdots,m_\mu,
\]
so that 
\begin{equation}\label{collectiveFrozenCoefficient}
    k_{\mu,\ell}^{\pm}(\delta,k_p)=\pm\sqrt\delta\sqrt\mu
    +\mathcal O\!\left(\sqrt\delta\,\bigl(\sqrt\delta+|k_p|\bigr)^{1/p_\mu}\right).
\end{equation}

\item[\textup{(b)}] 
If $0$ is an eigenvalue of $\mathfrak C$, the associated group consists of $2m_0$ characteristic values, counted with algebraic multiplicity. They can be labelled
\[
    k_{0,\ell}(\delta,k_p), \qquad \ell=1,\cdots,2m_0,
\]
and satisfy
\begin{equation}\label{collectiveZeroEigenvalueEstimate}
    \max_{1\leq\ell\leq2m_0}|k_{0,\ell}(\delta,k_p)| \leq K\sqrt\delta\, \bigl(\sqrt\delta+|k_p|\bigr)^{1/(2m_0)}.
\end{equation}
In particular, 
\[
    \max_{1\leq\ell\leq2m_0} |k_{0,\ell}(\delta,k_p)| = o(\sqrt{\delta})
    \qquad\text{as }(\sqrt\delta,k_p)\to(0,0).
\]
Every nonzero member of this group is a resonance of \eqref{nondfinal1}. 
\end{enumerate}

All estimates are uniform for sufficiently small $\delta>0$ and $|k_p|$, with $\tau,\gamma,k_s$ fixed.
\end{theorem}

\begin{proof}
For $q=1,\cdots,2N$, define
\begin{equation}\label{resonancePowerSums}
    s_q(\delta,k_p):=\frac1{2\pi\mathrm i}\operatorname{tr}
    \oint_{|k|=r}k^qF(k,\delta;k_p,k_s)^{-1}\partial_kF(k,\delta;k_p,k_s)\,dk,
\end{equation}
where the contour is positively oriented. Let $\zeta_1(\delta,k_p),\cdots,\zeta_{2N}(\delta,k_p)$
denote the characteristic values of $k\mapsto F(k,\delta;k_p,k_s)$ in $|k|<r$, listed with
algebraic multiplicity. By Proposition~\ref{prop:gohberg-sigal}(v),
\begin{equation}\label{resonancePowerSumIdentity}
    s_q(\delta,k_p)=\sum_{\nu=1}^{2N}\zeta_\nu(\delta,k_p)^q, \qquad q=1,\cdots,2N.
\end{equation}
The functions $s_q$ are jointly holomorphic in $(\delta,k_p)$, as follows from the local finite-dimensional reduction and the scalar logarithmic residue formula. Since all characteristic values are equal to zero at $\delta=0$, $s_q(0,k_p)=0$.

A direct computation yields
\[
    \partial_\delta(F^{-1}\partial_kF)=\partial_k(F^{-1}\partial_\delta F)
    +[F^{-1}\partial_kF,F^{-1}\partial_\delta F].
\]
Where $F$ is invertible, it has a pseudodifferential inverse of order $-1$. Since $\partial_\delta F$ has order $-2$, $F^{-1}\partial_\delta F$ has order $-3$ on the two-dimensional surface $\partial D$ and is trace class on $L^2(\partial D)$. Moreover, $F^{-1}\partial_kF$ is bounded. Therefore, cyclicity of the trace gives
\[
    \operatorname{tr}[F^{-1}\partial_kF,F^{-1}\partial_\delta F]=0.
\]
Differentiating \eqref{resonancePowerSums} and integrating by parts in $k$ gives
\[
\begin{aligned}
    \partial_\delta s_q(\delta,k_p)&=-\frac{q}{2\pi\mathrm i}\operatorname{tr}
    \oint_{|k|=r}k^{q-1}F(k,\delta;k_p,k_s)^{-1}\partial_\delta F(k,\delta;k_p,k_s)\,dk\\
    &=-\frac{q}{2\pi\mathrm i}\operatorname{tr}
    \oint_{|k|=r}k^{q-1}(I+\delta R(k,k_p))^{-1}R(k,k_p)\,dk,
\end{aligned}
\]
where $R(k,k_p)$ is defined in \eqref{collectiveRdefinition} and we used the affine dependence of $F$ on $\delta$. Repeated differentiation at $\delta=0$ yields
\begin{equation}\label{momentTaylorCoefficient}
    \frac{\partial_\delta^n s_q(0,k_p)}{n!}=\frac{(-1)^nq}{n}
    \operatorname{tr}\operatorname{Res}_{k=0}\left[k^{q-1}R(k,k_p)^n\right],\qquad n\geq1.
\end{equation}
%The traces in the contour and residue formulas are taken after integration or extraction of the residue; the resulting operators are finite rank.

To evaluate the leading residues, introduce the auxiliary matrix $C(k_p)=(C_{ij}(k_p))_{i,j=1}^N$ by
\[
    C_{ij}(k_p):=\frac{k_s^2}{\tau^2|D_j|}\int_{\partial D_j}
    \left(I-k_s^2\gamma\mathcal N_{k_p,k_s}\mathcal H\right)^{-1}\mathcal N_{k_p,k_s}[\mathbbm{1}_{\partial D_i}]\,dS.
\]
By Corollary~\ref{coroNkpks} and admissibility, this matrix is holomorphic for sufficiently small $|k_p|$, and
\[
    C(0)=\mathfrak C, \qquad C(k_p)=\mathfrak C+\mathcal O(|k_p|).
\]
Introduce
\[
    U:\mathbb C^N\rightarrow L^2(\partial D), \qquad Ua:=\sum_{j=1}^Na_j\varphi_j,
\]
and $V(k_p):L^2(\partial D)\to\mathbb C^N$ by
\begin{equation}\label{collectiveVdefinition}
    (V(k_p)h)_j:=\frac{k_s^2}{\tau^2|D_j|}\int_{\partial D_j}
    \left(I-k_s^2\gamma\mathcal N_{k_p,k_s}\mathcal H\right)^{-1} \mathcal N_{k_p,k_s}S[h]\,dS.
\end{equation}
By \eqref{TkInversePrincipalPart} and $S^{k\tau}=S+\mathcal O_{\mathcal L(L^2(\partial D),H^1(\partial D))}(|k|)$,
\begin{equation}\label{collectiveRprincipalPart}
    R(k,k_p)=-\frac{UV(k_p)}{k^2}+\mathcal O_{\mathcal L(L^2(\partial D))}(|k|^{-1}).
\end{equation}
Moreover, $S[\varphi_i]=\mathbbm{1}_{\partial D_i}$ implies that
\[
    V(k_p)U=C(k_p)^\top, \qquad V(0)U=\mathfrak C^\top.
\]
Consequently,
\[
    \operatorname{tr}_{L^2(\partial D)}\bigl((UV(k_p))^m\bigr)
    =\operatorname{tr}_{\mathbb C^N}\bigl((V(k_p)U)^m\bigr) =\operatorname{tr}\bigl(C(k_p)^m\bigr).
\]

By \eqref{collectiveRprincipalPart}, $R^n$ has a pole of order at most $2n$. Hence \eqref{momentTaylorCoefficient} vanishes when $q>2n$. For $q=2m$ and $n=m$,
\[
    \operatorname{Res}_{k=0}\left[k^{2m-1}R(k,k_p)^m\right]=(-1)^m(UV(k_p))^m.
\]
Thus, the Taylor coefficients of $s_{2m-1}$ and $s_{2m}$ of orders less than $m$ vanish, and
\[
    \frac{\partial_\delta^m s_{2m}(0,k_p)}{m!}=2\operatorname{tr}\bigl(C(k_p)^m\bigr).
\]
It follows that, for $m=1,\cdots,N$,
\begin{equation}\label{resonanceMomentExpansions}
\begin{aligned}
    s_{2m-1}(\delta,k_p)&=\mathcal O(|\delta|^m),\\
    s_{2m}(\delta,k_p)&=2\delta^m\operatorname{tr}\bigl(C(k_p)^m\bigr)+\mathcal O(|\delta|^{m+1}).
\end{aligned}
\end{equation}
The estimates are uniform for sufficiently small $|k_p|$, by joint holomorphy and Cauchy's estimates.

Set $\delta=t^2$. For $t\neq0$, define the monic polynomial
\[
    \mathscr P(\lambda,t,k_p):=\prod_{\nu=1}^{2N}\left(\lambda-\frac{\zeta_\nu(t^2,k_p)}{t}\right).
\]
The first $2N$ power sums of its roots are
\[
    p_q(t,k_p):=t^{-q}s_q(t^2,k_p),\qquad q=1,\cdots,2N.
\]
By \eqref{resonanceMomentExpansions}, these functions extend jointly holomorphically across $t=0$, with
\begin{equation}\label{scaledMomentLimits}
    p_{2m-1}(0,k_p)=0, \qquad p_{2m}(0,k_p)=2\operatorname{tr}\bigl(C(k_p)^m\bigr),
    \qquad m=1,\cdots,N.
\end{equation}
Newton's identities express the coefficients of $\mathscr P$ as polynomials in $p_1,\cdots,p_{2N}$. Therefore, these coefficients also extend jointly holomorphically across $t=0$.
The roots of $\det(\lambda^2I_N-C(k_p))$, counted with multiplicity, have precisely the power sums
\eqref{scaledMomentLimits}. Since these power sums uniquely determine a monic polynomial of degree $2N$, we obtain
\begin{equation}\label{collectiveLimitingPolynomial}
    \mathscr P(\lambda,0,k_p)=\det\bigl(\lambda^2I_N-C(k_p)\bigr).
\end{equation}
In particular,
\[
    \mathscr P(\lambda,0,0)=\det\bigl(\lambda^2I_N-\mathfrak C\bigr).
\]

We first partition the roots according to their limiting locations. Let $\sigma(\mathfrak C)$ denote the spectrum of $\mathfrak{C}$ and
\[
    \Lambda_0:=\{\alpha\in\mathbb C:\alpha^2\in\sigma(\mathfrak C)\}.
\]
For every distinct $\alpha\in\Lambda_0$, choose a small disc $\mathbb D_\alpha$ centered at $\alpha$, with pairwise disjoint closures. For $\alpha\neq0$, choose its radius less than
$|\alpha|/2$. Each disc contains only its center as a root of $\mathscr P(\lambda,0,0)$.
By the scalar Rouch\'e theorem, for sufficiently small $|t|$ and $|k_p|$, the number of roots of $\mathscr P(\lambda,t,k_p)$ in $\mathbb D_\alpha$, counted with multiplicity, is $m_{\alpha^2}$ if $\alpha\neq0$, and $2m_0$ if $\alpha=0$. These counts sum to $2N$, so all roots lie in these discs.
For $t\neq0$, this partitions the characteristic values $\zeta_\nu(t^2,k_p)$ by the locations of their ratios
$\zeta_\nu(t^2,k_p)/t$.

We next prove \textup{(a)}. Fix a nonzero eigenvalue $\mu\in\sigma(\mathfrak C)$ and let $\alpha$ be either
$\sqrt\mu$ or $-\sqrt\mu$. Write $t=\sqrt\delta>0$. For $\lambda\in\overline{\mathbb D_\alpha}$,
define
\[
    H_0(\lambda):=I-\lambda^{-2}UV(0).
\]
Since this disc is bounded away from zero, \eqref{normalizedCollectiveF}, \eqref{collectiveRprincipalPart},
and the holomorphy of $V$ give
\begin{equation}\label{semisimpleNormalizedComparison}
\begin{aligned}
    \widehat F(t\lambda,t^2;k_p,k_s)&=H_0(\lambda)+\mathscr E(\lambda,t,k_p),\\
    \|\mathscr E(\lambda,t,k_p)\|_{\mathcal L(L^2(\partial D))}&\leq C_E(t+|k_p|),
\end{aligned}
\end{equation}
uniformly on $\overline{\mathbb D_\alpha}$. Since $V(0)U=\mathfrak C^\top$, direct multiplication gives
\[
    H_0(\lambda)^{-1}=I+U\left(\lambda^2I_N-\mathfrak C^\top\right)^{-1}V(0),
    \qquad \lambda\in\overline{\mathbb D_\alpha}\setminus\{\alpha\}.
\]
For a Jordan block $J_d(\mu)=\mu I_d+N_d$, where $N_d^d=0$,
\[
    (\lambda^2I_d-J_d(\mu))^{-1}=\sum_{j=0}^{d-1}\frac{N_d^j}{(\lambda^2-\mu)^{j+1}}.
\]
All such blocks have size at most $p_\mu$, while the resolvents of blocks associated with the other eigenvalues are uniformly bounded on $\overline{\mathbb D_\alpha}$. Since $\lambda^2-\mu=(\lambda-\alpha)(\lambda+\alpha)$ and $\lambda+\alpha$ is bounded away from zero on this disc,
\begin{equation}\label{limitingResolventEstimate}
    \|H_0(\lambda)^{-1}\|_{\mathcal L(L^2(\partial D))}
    \leq\frac{C_R}{|\lambda-\alpha|^{p_\mu}},
    \qquad
    \lambda\in\overline{\mathbb D_\alpha}\setminus\{\alpha\}.
\end{equation}

Let $k_*$ be any characteristic value with $\lambda_*:=k_*/t\in\mathbb D_\alpha$.
Since $F(k_*,\delta;k_p,k_s)$ is Fredholm of index zero and not invertible, choose
\[
    \varphi_*\in\ker F(k_*,\delta;k_p,k_s), \qquad \|\varphi_*\|_{L^2(\partial D)}=1.
\]
As $0\notin\overline{\mathbb D_\alpha}$, we have $k_*\neq0$, and $F(k_*,0;k_p,k_s)$ is invertible. Hence,
\[
    \widehat F(t\lambda_*,t^2;k_p,k_s)\varphi_*=0,
\]
and \eqref{semisimpleNormalizedComparison} yields
\[
    H_0(\lambda_*)\varphi_*=-\mathscr E(\lambda_*,t,k_p)\varphi_*.
\]
If $\lambda_*\neq\alpha$, applying $H_0(\lambda_*)^{-1}$ and using \eqref{limitingResolventEstimate} gives
\begin{equation}
    1\leq\|H_0(\lambda_*)^{-1}\|\|\mathscr E(\lambda_*,t,k_p)\|\leq\frac{C_RC_E(t+|k_p|)}
    {|\lambda_*-\alpha|^{p_\mu}}.
\end{equation}
Therefore,
\begin{equation}\label{collectiveRootLocalization}
    \left|\frac{k_*}{\sqrt\delta}-\alpha\right|
    \leq K_\mu\bigl(\sqrt\delta+|k_p|\bigr)^{1/p_\mu}.
\end{equation}
The same estimate is immediate when $\lambda_*=\alpha$.

The preceding root count gives $m_\mu$ characteristic values for each choice $\alpha=\pm\sqrt\mu$, with algebraic
multiplicities. Labelling them as $k_{\mu,\ell}^{\pm}(\delta,k_p)$ and multiplying
\eqref{collectiveRootLocalization} by $\sqrt\delta$ proves \eqref{collectiveFrozenCoefficient}.
They are nonzero resonances by Proposition~\ref{prop:boundary-reduction}. This proves \textup{(a)}.

Finally, we prove \textup{(b)}. Suppose that
$0\in\sigma(\mathfrak C)$. Then
\[
    \mathscr P(\lambda,0,0)=\lambda^{2m_0}Q_0(\lambda),\qquad Q_0(0)\neq0.
\]
The disc $\mathbb D_0$ can be chosen so that
\[
    |Q_0(\lambda)|\geq c_0>0,\qquad \lambda\in\overline{\mathbb D_0}.
\]
Joint holomorphy of the coefficients of $\mathscr P$ gives
\[
    \sup_{\lambda\in\overline{\mathbb D_0}}
    |\mathscr P(\lambda,t,k_p)-\mathscr P(\lambda,0,0)|\leq K_0(|t|+|k_p|).
\]
By the preceding root count, $\mathbb D_0$ contains exactly $2m_0$ roots of $\mathscr P(\lambda,t,k_p)$, counted with multiplicity. For $t=\sqrt\delta>0$, denote the corresponding characteristic values by $k_{0,\ell}(\delta,k_p)$ and set
\[
    \lambda_{0,\ell}:=\frac{k_{0,\ell}(\delta,k_p)}{t},\qquad \ell=1,\cdots,2m_0.
\]
Since $\mathscr P(\lambda_{0,\ell},t,k_p)=0$,
\[
\begin{aligned}
    c_0|\lambda_{0,\ell}|^{2m_0}&\leq |\lambda_{0,\ell}^{2m_0}Q_0(\lambda_{0,\ell})|\\
    &=|\mathscr P(\lambda_{0,\ell},0,0)-\mathscr P(\lambda_{0,\ell},t,k_p)|\\
    &\leq K_0(t+|k_p|).
\end{aligned}
\]
Consequently,
\[
    \max_{1\leq\ell\leq2m_0}|k_{0,\ell}(\delta,k_p)|\leq K\sqrt\delta\,
    \bigl(\sqrt\delta+|k_p|\bigr)^{1/(2m_0)},
\]
which proves \eqref{collectiveZeroEigenvalueEstimate}. Every nonzero member of this group is a resonance by
Proposition~\ref{prop:boundary-reduction}. This proves \textup{(b)}.
\end{proof}

We conclude this section by restoring the physical frequency dependence of $k$, $k_p$, $k_s$, and $\gamma$.
To display the relevant parameter dependence, write $\mathfrak C(\tau,\gamma,k_s)$ for the matrix defined in
\eqref{collectiveCapacitanceMatrixkbzero}. Since $\gamma$ and $k_s$ depend on $\omega$, the leading-order
resonance condition becomes an implicit equation for the physical frequency.
We formulate this equation through a determinant. The following corollary shows that, under uniform admissibility
and nondegeneracy assumptions, an isolated nonzero root of this equation gives rise to a cluster of nearby physical resonances with the same total algebraic multiplicity.

\begin{corollary}\label{cor:physical-resonance-frequencies}
Let $\delta, \tau, k(\omega), \gamma(\omega),k_p(\omega)$, and $k_s(\omega)$ be given in Section \ref{sec:1}.
In particular, $\delta = \rho_b/\rho$ and $k(\omega)=\omega/\Omega_*$ for 
\[
    \Omega_*:=\frac1a\sqrt{\frac\kappa\rho}.
\]
Define the normalized leading-order determinant
\begin{equation}\label{physical-leading-determinant}
    \mathfrak d_\delta(\omega):=\det\left(
        \frac{k(\omega)^2}{\delta}I_N -\mathfrak C\bigl(\tau,\gamma(\omega),k_s(\omega)\bigr)
    \right).
\end{equation}
Suppose that $\omega_b\neq0$ is a zero of $\mathfrak d_\delta$ of order $m_b\geq 1$ and that
\begin{equation}\label{physical-root-nondegeneracy}
\left\{\begin{aligned}
    & c_b\sqrt\delta\leq |k(\omega_b)|\leq C_b\sqrt\delta ,\\
    & \frac{|\omega_b|^{m_b}}{m_b!}\left|\mathfrak d_\delta^{(m_b)}(\omega_b)\right|\geq c_0>0
\end{aligned}\right.
\end{equation}
for some $c_b,C_b,c_0>0$. Assume that, on the closed disc
\[
    \overline{\mathbb D}_\theta:=\{\omega\in\mathbb C:|\omega-\omega_b|\leq\theta|\omega_b|\},
    \qquad 0<\theta<1,
\]
$k_p$, $k_s$, and $\gamma$ are holomorphic, and $(\gamma(\omega),k_s(\omega))$ belongs to a fixed compact subset of the admissible parameter set. Here $m_b$, $c_b$, $C_b$, $c_0$, $\theta$, and the compact
admissible set are independent of $\delta$.

Then there exists $0<\eta<\theta$, independent of $\delta$, such that, for sufficiently small $\delta$, the physical problem \eqref{maineq} has exactly $m_b$ resonances in $|\omega-\omega_b|<\eta|\omega_b|$,
counted with algebraic multiplicity. They satisfy
\begin{equation}\label{physical-resonance-cluster-asymptotics}
    \omega_{\mathrm{res},\ell}=\omega_b+
    \mathcal O\!\left(\frac1a\sqrt{\frac\kappa\rho}\,\delta^{\frac12+\frac1{2m_b}}\right),
    \qquad \ell=1,\cdots,m_b.
\end{equation}
\end{corollary}
\begin{proof}
Set $t=\sqrt\delta>0$. The physical parameter relations give
\[
    k_p(\omega)^2=\frac{\rho a^2\omega^2}{\kappa+\mathrm i\omega(\lambda+2\mu)}.
\]
Let $\omega_*:=\mathrm i\kappa/(\lambda+2\mu)$ be the pole of $k_p(\omega)^2$. Since $k_p$ is holomorphic near $\overline{\mathbb D}_\theta$, $|\omega_*-\omega_b|>\theta|\omega_b|$.
For $|\omega-\omega_b|\leq(\theta/2)|\omega_b|$, it follows that
\[
    |\omega-\omega_*|\geq\frac12|\omega_b-\omega_*|,
    \qquad |\omega_*|\leq(1+\theta^{-1})|\omega_b-\omega_*|.
\]
Consequently,
\[
    |\kappa+\mathrm i \omega(\lambda+2\mu)|=\kappa\frac{|\omega-\omega_*|}{|\omega_*|}
    \geq\frac{\theta}{2(1+\theta)}\kappa.
\]
Together with the first line of \eqref{physical-root-nondegeneracy}, this gives
\begin{equation}\label{physical-compressional-estimate}
    |k_p(\omega)|\leq C|k(\omega)|\leq C|k(\omega_b)|\leq Ct
    \qquad \text{for }|\omega-\omega_b|\leq(\theta/2)|\omega_b|.
\end{equation}
Thus, Theorem~\ref{thm314} applies on this smaller disc $\overline{\mathbb D}_{\theta/2}$ when $t$ is sufficiently small.

We use the polynomial $\mathscr P$ constructed in the proof of Theorem~\ref{thm314}. Displaying its dependence on the remaining parameters, that construction gives
\[
    \mathscr P(\lambda,0,0;\tau,\gamma,k_s)=\det(\lambda^2I_N-\mathfrak C(\tau,\gamma,k_s)).
\]
Its coefficients are jointly holomorphic also in $(\tau,\gamma,k_s)$. Therefore, for $\lambda$ in any fixed bounded set,
\[
    \mathscr P(\lambda,t,k_p;\tau,\gamma,k_s)=\det(\lambda^2I_N-\mathfrak C(\tau,\gamma,k_s))+\mathcal O(|t|+|k_p|),
\]
uniformly on the parameter range under consideration. Define
\[
    \mathfrak d_\delta^{\mathrm{ex}}(\omega)
    :=\mathscr P\left(\frac{k(\omega)}{t},t,k_p(\omega);\tau,\gamma(\omega),k_s(\omega)\right).
\]
This function is holomorphic near $\overline{\mathbb D}_{\theta/2}$. By the first line of \eqref{physical-root-nondegeneracy} and \eqref{physical-compressional-estimate},
\begin{equation}\label{physical-determinant-comparison}
    \mathfrak d_\delta^{\mathrm{ex}}(\omega)=\mathfrak d_\delta(\omega)+\mathcal O(t).
\end{equation}

For fixed $t\neq0$, the zeros of $k\mapsto t^{2N}\mathscr P(k/t,t,k_p;\tau,\gamma,k_s)$
are precisely the characteristic values of $F$ in $|k|<r$, with their algebraic multiplicities. Since $|k(\omega)|=\mathcal O(t)$ and $\omega\neq0$ on the disc, Proposition~\ref{prop:boundary-reduction}
shows that the zeros of $\mathfrak d_\delta^{\mathrm{ex}}$ are exactly the physical resonances, with algebraic multiplicity. Here, the physical multiplicity is that of the frequency-dependent
Fredholm family after all physical substitutions.

Consider the relative coordinate $z=(\omega-\omega_b)/\omega_b$.
Compactness and the first line of \eqref{physical-root-nondegeneracy} give a uniform
bound for
\[
    z\mapsto\mathfrak d_\delta\bigl(\omega_b(1+z)\bigr)\qquad\text{on }|z|\leq\theta.
\]
Since $\omega_b$ is a zero of order $m_b$, Cauchy's estimates yield
\[
    \mathfrak d_\delta\bigl(\omega_b(1+z)\bigr)=a_\delta z^{m_b}+\mathcal O(|z|^{m_b+1}),
    \qquad a_\delta:=\frac{\omega_b^{m_b}}{m_b!}\mathfrak d_\delta^{(m_b)}(\omega_b),
\]
with a uniform remainder for $|z|\leq\theta/2$. By \eqref{physical-root-nondegeneracy}, $|a_\delta|\geq c_0$.
We can therefore choose $0<\eta<\theta/2$, independently of $\delta$, such that
\begin{equation}\label{physical-leading-lower-bound}
    |\mathfrak d_\delta(\omega)|\geq\frac{c_0}{2}\left|\frac{\omega-\omega_b}{\omega_b}\right|^{m_b}
    \qquad \text{for }|\omega-\omega_b|\leq\eta|\omega_b|.
\end{equation}
In particular, $\omega_b$ is the only zero of $\mathfrak d_\delta$ in this disc and has multiplicity $m_b$.
On its boundary, \eqref{physical-determinant-comparison} is smaller than the lower bound in \eqref{physical-leading-lower-bound} for sufficiently small $t$. Rouch\'e's theorem implies that
$\mathfrak d_\delta^{\mathrm{ex}}$ also has exactly $m_b$ zeros there, counted with multiplicity.

At any such zero $\omega_{\mathrm{res},\ell}$, \eqref{physical-determinant-comparison} and
\eqref{physical-leading-lower-bound} give
\[
    \frac{c_0}{2}\left|\frac{\omega_{\mathrm{res},\ell}-\omega_b}{\omega_b}\right|^{m_b}
    \leq |\mathfrak d_\delta(\omega_{\mathrm{res},\ell})|\leq Ct.
\]
Hence,
\[
    |\omega_{\mathrm{res},\ell}-\omega_b|\leq C|\omega_b|t^{1/m_b}\leq C\Omega_*\delta^{\frac12+\frac1{2m_b}},
\]
which proves \eqref{physical-resonance-cluster-asymptotics}.
\end{proof}

\begin{remark}
    Near the physical frequency $\omega_b$, the interior pressure and therefore the scattering field can be enhanced. We give a short explanation here. Write
\[
    \mathcal F_{\mathrm{phys}}(\omega):=F\bigl(k(\omega),\delta;k_p(\omega),k_s(\omega)\bigr).
\]
After subtracting an incident-field lifting, let the boundary equation for the interior density be
\[
    \mathcal F_{\mathrm{phys}}(\omega)\varphi(\omega)=h_{\mathrm{inc}}(\omega),
\]
where $h_{\mathrm{inc}}$ is the analytic boundary forcing determined by the prescribed incident field. 
Let $u_*\neq0$ span $\ker\mathcal F_{\mathrm{phys}}(\omega_*)$, and let the nonzero functional $\ell_*$ annihilate its range. Algebraic simplicity implies
\[
    a_*:=\ell_*\bigl(\mathcal F_{\mathrm{phys}}'(\omega_*)u_*\bigr)\neq0.
\]
Proposition \ref{prop:gohberg-sigal}(ii) then gives
\[
    \mathcal F_{\mathrm{phys}}(\omega)^{-1}=\frac{u_*\otimes\ell_*}{(\omega-\omega_*)a_*}+\mathcal H_*(\omega),
\]
where $\mathcal H_*$ is holomorphic. Hence, the singular part of the interior pressure is
\[
    p(\omega,x)=\frac{\ell_*(h_{\mathrm{inc}}(\omega_*))}{(\omega-\omega_*)a_*}S^{\tau k(\omega_*)}[u_*](x).
\]
When $\omega$ closes to $\omega_*$, the factor $(\omega-\omega_*)^{-1}$ above implies that the pressure is enhanced.
\end{remark}

\section{Example: spherical bubble}\label{sec:3}

Let the physical bubble be the ball $B_a$ of radius $a$. We choose $a$ as the reference length in \eqref{rescaleD}, so that all dimensionless boundary operators in this section act on
$\partial B_1$. We compute the scalar matrix $\mathfrak C$ and apply Theorem~\ref{thm314} and
Corollary~\ref{cor:physical-resonance-frequencies}.

\begin{lemma}\label{Lem41}
For $D=B_1$ and $k_s\in\mathbb C^*\setminus\mathcal E$,
\begin{equation}\label{computeN}
    \mathcal N_{0,k_s}[1]=\frac1{k_s^2-4}.
\end{equation}
\end{lemma}

\begin{proof}
By Lemma~\ref{lemma35vkpks}, the outgoing solution of \eqref{uindefDtN} with $k_p=0$ and $g=1$ is unique.
Rotational invariance therefore makes the solution radial. The divergence-free condition and the pressure equation give
\[
    \mathbf v(x)=M\frac{x}{|x|^3},\qquad q(x)=-\frac{k_s^2M}{|x|},\qquad |x|>1.
\]
On $\partial B_1$, these fields satisfy
\[
    \mathbf v\cdot\mathbf n=M,\qquad\mathbf t_{\mathbf n}(q,\mathbf v)=(k_s^2-4)M\mathbf n.
\]
The prescribed traction thus gives $M=(k_s^2-4)^{-1}$, which proves \eqref{computeN}.
\end{proof}

Since $\mathcal H[1]=-2$, Lemma~\ref{Lem41} implies that, for an admissible pair $(\gamma,k_s)$,
\[
    \left(I-k_s^2\gamma\mathcal N_{0,k_s}\mathcal H\right)^{-1}\mathcal N_{0,k_s}[1]
    =\frac1{k_s^2(1+2\gamma)-4}.
\]
Consequently, \eqref{collectiveCapacitanceMatrixkbzero} reduces to
\begin{equation}\label{sphericalCapacitance}
    \mathfrak C=\mathfrak c_{\mathrm{sph}}(\tau,\gamma,k_s):=\frac{3k_s^2}
    {\tau^2\bigl(k_s^2(1+2\gamma)-4\bigr)}.
\end{equation}
Its denominator is nonzero by admissibility, and its numerator is nonzero because $k_s\neq0$. Thus its only eigenvalue is nonzero and simple. Theorem~\ref{thm314}(a) immediately gives exactly two simple Minnaert-type resonances near zero:
\begin{equation}\label{resonanceforsphere}
    k_\pm(\delta,k_p)=\pm\sqrt\delta \sqrt{\frac{3k_s^2}{\tau^2\bigl(k_s^2(1+2\gamma)-4\bigr)}}
    +\mathcal O\!\left(\delta+\sqrt\delta\,|k_p|\right).
\end{equation}
In particular, the remainder is $\mathcal O(\delta)$ when $k_p=0$.

We now restore the physical frequency dependence. Set
\[
    \Omega_*:=\frac1a\sqrt{\frac\kappa\rho}, \qquad
    H(\omega):=\omega^2-\frac{4\mathrm i\mu}{\rho a^2}\omega+\frac{2\sigma}{\rho a^3},
    \qquad G:=\frac{3\kappa_b}{\rho a^2}.
\]
By \eqref{nondmulambda}, \eqref{nondtauk}, and \eqref{defkpks},
\[
    k_s(\omega)^2=\frac{\rho a^2\omega}{\mathrm i\mu},
    \qquad \gamma(\omega)=\frac{\sigma}{\rho a^3\omega^2},
    \qquad G=\frac{3\Omega_*^2\delta}{\tau^2}.
\]
Substituting these relations into \eqref{sphericalCapacitance} yields
\begin{equation}\label{sphericalPhysicalCoefficient}
    \mathfrak c_{\mathrm{sph}}\bigl(\tau,\gamma(\omega),k_s(\omega)\bigr)
    =\frac{3\omega^2}{\tau^2H(\omega)}.
\end{equation}
Hence, the normalized determinant \eqref{physical-leading-determinant} takes the scalar form
\begin{equation}\label{sphericalPhysicalDeterminant}
    \mathfrak d_\delta(\omega)=\frac{\omega^2}{\Omega_*^2\delta}\frac{Q(\omega)}{H(\omega)},
    \qquad Q(\omega):=H(\omega)-G.
\end{equation}
At nonzero frequencies in the admissible parameter range, its zeros are therefore exactly the roots of
\begin{equation}\label{spherical-physical-quadratic}
    Q(\omega)=\omega^2-\frac{4\mathrm i\mu}{\rho a^2}\omega+\frac{2\sigma}{\rho a^3}-\frac{3\kappa_b}{\rho a^2}=0.
\end{equation}
These roots are
\begin{equation}\label{physicalballformula}
    \omega_M^\pm:=\frac{2\mathrm i\mu}{\rho a^2}
    \pm\frac1a\sqrt{\frac{3\kappa_b}{\rho}-\frac{2\sigma}{\rho a}-\frac{4\mu^2}{\rho^2a^2}}.
\end{equation}

\begin{corollary}[Physical resonances of a spherical bubble]\label{cor:spherical-physical-resonances}
Let $\omega_M$ be either nonzero root in \eqref{physicalballformula}. Assume that the analytic-domain and admissibility hypotheses of Corollary~\ref{cor:physical-resonance-frequencies} hold on a
fixed relative disc centered at $\omega_M$, and that
\begin{equation}\label{spherical-physical-simple-roots}
    \left|2\omega_M-\frac{4\mathrm i\mu}{\rho a^2}\right|\geq m|\omega_M|,
\end{equation}
where $m>0$ is independent of $\delta$. For sufficiently small $\delta$, there is exactly one physical
resonance in a sufficiently small fixed relative neighborhood of $\omega_M$. It is simple and satisfies
\begin{equation}\label{spherical-physical-resonance-asymptotics}
    \omega_{\mathrm{res}}=\omega_M+\mathcal O\!\left(\frac1a\sqrt{\frac\kappa\rho}\,\delta\right).
\end{equation}
\end{corollary}

\begin{proof}
The scalar in \eqref{sphericalCapacitance} is continuous and nonzero on the prescribed compact admissible set. It is therefore uniformly bounded above and away from zero, so the leading equation gives
\[
    |\omega_M|\asymp\Omega_*\sqrt\delta.
\]
Thus, the first line of \eqref{physical-root-nondegeneracy} holds. Since $H(\omega_M)=G$, differentiation of
\eqref{sphericalPhysicalDeterminant} gives
\[
    \mathfrak d_\delta'(\omega_M)=\frac{\omega_M^2}{\Omega_*^2\delta\,G}
      \left(2\omega_M-\frac{4\mathrm i\mu}{\rho a^2}\right).
\]
In particular,
\[
    |\omega_M|\,|\mathfrak d_\delta'(\omega_M)|=\frac{\tau^2}{3}
      \left|\frac{\omega_M}{\Omega_*\sqrt\delta}\right|^4
      \frac{\left|2\omega_M-4\mathrm i\mu/(\rho a^2)\right|}{|\omega_M|}.
\]
The scale estimate, the compactness assumption on $\tau>0$, and \eqref{spherical-physical-simple-roots} verify
the second line of \eqref{physical-root-nondegeneracy} with $m_b=1$. All hypotheses of Corollary~\ref{cor:physical-resonance-frequencies} now hold with
$\omega_b=\omega_M$ and $m_b=1$, giving the conclusion. 
\end{proof}

In the underdamped case, the radicand in \eqref{physicalballformula} is positive. With the convention $\mathrm e^{\mathrm i\omega t}$, its imaginary part $2\mu/(\rho a^2)$ is the viscous decay rate
\cite[Eq.~(91)]{review-bubbles}. The positive real part agrees with the damped natural frequency in
\cite[Eq.~(47)]{review-bubbles}, with the gas bulk modulus identified as $\kappa_b=\varkappa p_{g,0}$, where $\varkappa$ is the real polytropic exponent and $p_{g,0}$ the equilibrium gas pressure.
Formally, setting $\mu=\sigma=0$ in \eqref{physicalballformula} recovers the classical leading Minnaert frequency
$a^{-1}\sqrt{3\kappa_b/\rho}$ \cite{ammari2018minnaert}.

\section{Discussion: possible extension to higher-order acoustic modes}
\label{sec:5}

The analysis in Section~\ref{sec:2} concerns resonances emerging from the zero interior Neumann eigenvalue. For the inviscid acoustic problem, resonances associated with nonzero Neumann eigenvalues
are studied in \cite{fabry3d}. We outline the corresponding formal calculation for the present viscous model with surface tension.

Let $D$ be connected and let $\nu_0>0$ be a simple Neumann eigenvalue of $-\Delta$ in the dimensionless domain $D$. Choose a real eigenfunction $p_0$ normalized by $\int_Dp_0^2\,dx=1$, and write
\[
    k_0:=\frac{\sqrt{\nu_0}}{\tau}, \qquad \omega_0:=\Omega_*k_0,
    \qquad \Omega_*:=\frac1a\sqrt{\frac\kappa\rho}.
\]
Thus,
\begin{equation}\label{higherNeumannMode}
    (\Delta+\tau^2k_0^2)p_0=0\quad\text{in }D, \qquad \partial_{\mathbf n}p_0=0\quad\text{on }\partial D.
\end{equation}

Use the outgoing branches from Section~\ref{secradiationcind}. Assume locally near $\omega_0$ that the exterior problem defining $\mathcal N_{k_p(\omega),k_s(\omega)}$ is uniquely solvable and that this operator depends holomorphically on $\omega$. Assume also that
\[
    I-k_s(\omega)^2\gamma(\omega)\mathcal N_{k_p(\omega),k_s(\omega)}\mathcal H
    :L^2(\partial D)\rightarrow H^{-1}(\partial D)
\]
is invertible there. These are local higher-frequency assumptions; the small-$k_p$ results of Section~\ref{sec:2} do not by themselves verify them at an arbitrary $\omega_0$.
Define
\begin{equation}\label{higherEffectiveBoundaryMap}
    \mathcal Q(\omega):=\left(I-k_s(\omega)^2\gamma(\omega)
        \mathcal N_{k_p(\omega),k_s(\omega)}\mathcal H\right)^{-1}
      \mathcal N_{k_p(\omega),k_s(\omega)}.
\end{equation}

The interface conditions in \eqref{nondfinal1} imply
\[
    \left(I-k_s(\omega)^2\gamma(\omega)\mathcal N_{k_p(\omega),k_s(\omega)}\mathcal H\right)
        \partial_{\mathbf n}p
    =-\delta k_s(\omega)^2\mathcal N_{k_p(\omega),k_s(\omega)}[p|_{\partial D}].
\]
Consequently, the interior pressure formally satisfies
\begin{equation}\label{higherEffectiveBoundaryProblem}
\left\{
\begin{aligned}
    &(\Delta+\tau^2k(\omega)^2)p=0 &&\text{in }D,\\
    &\partial_{\mathbf n}p=-\delta k_s(\omega)^2\mathcal Q(\omega)[p|_{\partial D}]
      &&\text{on }\partial D.
\end{aligned}
\right.
\end{equation}

Seek formal expansions
\[
    k(\omega_\delta)=k_0+\delta k_1+\cdots, \qquad p_\delta=p_0+\delta p_1+\cdots.
\]
The terms of order $\delta$ in \eqref{higherEffectiveBoundaryProblem} give
\[
\left\{
\begin{aligned}
    &(\Delta+\tau^2k_0^2)p_1=-2\tau^2k_0k_1p_0&&\text{in }D,\\
    &\partial_{\mathbf n}p_1=-k_s(\omega_0)^2\mathcal Q(\omega_0)[p_0|_{\partial D}]&&\text{on }\partial D.
\end{aligned}
\right.
\]
Green's identity and \eqref{higherNeumannMode} then give
\[
    -2\tau^2k_0k_1=\int_D\bigl(p_0\Delta p_1-p_1\Delta p_0\bigr)\,dx
    =\int_{\partial D}p_0\partial_{\mathbf n}p_1\,dS.
\]
Thus, the first-order coefficient is
\begin{equation}\label{higherorderforma}
\begin{aligned}
    k_1={}&\frac{k_s(\omega_0)^2}{2\tau^2k_0}\int_{\partial D}p_0
    \left(I-k_s(\omega_0)^2\gamma(\omega_0)\mathcal N_{k_p(\omega_0),k_s(\omega_0)}\mathcal H\right)^{-1}\\
    &\hspace{28mm}\Big[\mathcal N_{k_p(\omega_0),k_s(\omega_0)}[p_0|_{\partial D}]\Big]\,dS.
\end{aligned}
\end{equation}
Equivalently, the formal physical-frequency correction is $\omega_\delta=\omega_0+\delta\Omega_*k_1+\cdots$.

Unlike the Minnaert branches, this expansion is centered at a nonzero frequency. The linear term in
$k(\omega_\delta)^2-k_0^2$ is $2\delta k_0k_1$, so the shift is of order $\delta$, rather than $\sqrt\delta$.

% The exterior coefficients in the boundary correction can be evaluated
% at $\omega_0$ because they are multiplied by $\delta$; their frequency
% variation contributes only to subsequent formal orders.
% This does not justify evaluating $k_s(\omega)$ or $\gamma(\omega)$
% at zero in the Minnaert calculation, where the physical dependence
% must instead be retained as in
% Corollary~\ref{cor:physical-resonance-frequencies}.

For a Neumann eigenvalue of multiplicity $d$, let $u_1,\cdots,u_d$ be a real orthonormal basis of its eigenspace.
The same formal projection replaces the scalar coefficient by
\[
    \mathsf M_{\alpha\beta}:=k_s(\omega_0)^2
      \int_{\partial D}u_\alpha\mathcal Q(\omega_0)[u_\beta|_{\partial D}]\,dS,
    \qquad \alpha,\beta=1,\cdots,d.
\]
The first-order compatibility condition becomes $(2\tau^2k_0k_1I_d-\mathsf M)c=0$ for a nonzero vector of
mode amplitudes $c$. This is analogous to the finite-dimensional treatment in
\cite{fabry3d}, but a rigorous viscous extension requires a separate
analysis of the local exterior solvability, spectral reduction,
and remainder bounds.

\section*{Acknowledgments}
  The work of W.J.\,is partially supported by the NSFC Grant No.\,12571220 and by the New Cornerstone Investigator Program 100001127. The work of H.L.\,is substantially supported by the NSFC Grant No.\,12401561. This paper was initiated while H.A.\,was visiting the Hong Kong Institute for Advanced Study as a Senior Fellow.

\appendix

\section{Derivation of the governing equations}\label{appenequ}

We briefly derive the governing equations for immiscible two-phase flow with surface tension. Let $V_f$ and $V_g$ denote the regions occupied by two phases $f$ and $g$, respectively, separated by an interface $\Sigma$. For each $k\in\{f,g\}$, the conservation of mass and momentum gives
\begin{equation}\label{massmomequa}
\begin{cases}
\partial_t\rho_k+\nabla\cdot(\rho_k\mathbf{v}_k)=0,\\
\partial_t(\rho_k\mathbf{v}_k)
+\nabla\cdot(\rho_k\mathbf{v}_k\otimes\mathbf{v}_k)
=\nabla\cdot\mathbf{T}_k,
\end{cases}
\end{equation}
where, for Newtonian fluids,
\begin{equation}\label{Cauchystress}
\mathbf{T}_k=-p_k\mathbf{I}+2\mu_k\nabla^{\mathrm{s}}\mathbf{v}_k
+\lambda_k(\nabla\cdot\mathbf{v}_k)\mathbf{I}.
\end{equation}
Here, $\rho_k$, $p_k$, and $\mathbf{v}_k$ denote the density, pressure, and velocity, respectively, and $\mu_k$ and $\lambda_k$ are the dynamic and the second viscosity coefficients of phase $k$.

We describe the interface conditions. Let $\mathbf{n} $ denote the unit normal along $\Sigma$ pointing from phase $g$ to phase $f$. The immiscible condition says that there is no phase change across the interface $\Sigma$. Thus, $\mathbf{v}_f \cdot \mathbf{n} = V_\Sigma = \mathbf{v}_g \cdot \mathbf{n} $, where $V_\Sigma$ denotes the normal velocity of the interface.

To derive the stress balance across $\Sigma$, consider a patch $S\subset \Sigma$ on the interface, which can be thought of as the zero thickness limit of a volume of the two-phase body with $S$ in the center. For each $p\in \partial S$, let $\mathbf{s}_p$ denote the outward unit conormal to $\partial S$ within $\Sigma$, and let $\mathbf{t}_p\in T_p\partial S$ be the unit tangent vector oriented so that $\mathbf{s}_p \times \mathbf{t}_p = \mathbf{n}_p$.
The total force on $S$ is zero because the surface has zero density in the zero-thickness limit, \emph{i.e.},
\begin{equation*}
    \int_S (\mathbf{T}_f\mathbf{n} - \mathbf{T}_g\mathbf{n} )\,dA + \oint_{\partial S} \sigma \mathbf{s}\,d\ell = 0.
\end{equation*}
Here, $dA$ denotes the surface measure on $S$ and $d\ell$ denotes the length measure on $\partial S$, and $\sigma>0$ denotes the surface-tension coefficient and is assumed to be constant. The first integral represents the jump in fluid traction across the interface, while the second accounts for surface tension. By the surface divergence theorem,
\begin{equation*}
\begin{aligned}
\oint_{\partial S} \sigma \mathbf{s} \,d\ell =-\sigma \int_S  (\nabla\cdot \mathbf{n}) \mathbf{n}\,dA =2\sigma \int_S H \mathbf{n}\,dA ,
\end{aligned}
\end{equation*}
where $H=-\frac12 \nabla\cdot \mathbf{n}$ is the signed mean curvature of $\Sigma$. Since $S$ is arbitrary, we get the pointwise balance condition 
\begin{equation}\label{boundmoment}
   (\mathbf{T}_g - \mathbf{T}_f)\mathbf{n} = 2\sigma H\mathbf{n}  \qquad \text{on } \Sigma. 
\end{equation}

Assuming a barotropic pressure law $p_k=p_k(\rho_k)$ for each phase, we obtain the nonlinear two-phase flow system
\begin{equation}\label{fullequation}
    \left\{
\begin{aligned}
    & \partial_t \rho_k + \nabla \cdot (\rho_k \mathbf{v}_k) = 0, &&  \\
    & \partial_t (\rho_k\mathbf{v}_k) + \nabla \cdot (\rho_k \mathbf{v}_k\otimes \mathbf{v}_k) = \nabla \cdot \mathbf{T}_k, &&  \\
    & p_k = p_k(\rho_k), &&  \\
    & (\mathbf{v}_f - \mathbf{v}_g ) \cdot \mathbf{n} =0 \qquad\qquad\, \mathrm{on} \ \Sigma , &&  \\
    & (\mathbf{T}_g - \mathbf{T}_f ) \mathbf{n} = 2\sigma H\mathbf{n} \qquad\  \mathrm{on} \ \Sigma.
\end{aligned}
    \right.
\end{equation}

Following the classical models of small-amplitude bubble oscillations \cite{add1,add2,add3}, we linearize system \eqref{fullequation} about a stationary reference state with constant densities $\rho_k^0$, constant pressures $p_k^0$, and vanishing velocities. We write the perturbed fields as
\begin{equation}\label{perturbedfield}
    \rho_k^\varepsilon=\rho_k^0+\varepsilon\rho_k+O(\varepsilon^2),
    \qquad p_k^\varepsilon=p_k^0+\varepsilon p_k+O(\varepsilon^2),
    \qquad \mathbf{v}_k^\varepsilon=\varepsilon\mathbf{v}_k+O(\varepsilon^2),
\end{equation}
where $\varepsilon>0$ is a small parameter and $\rho_k$, $p_k$, and $\mathbf{v}_k$ now denote the first-order perturbations. The corresponding stress tensor takes the following form:
\begin{equation}\label{perturbT}
    (\mathbf{T}_k)_\varepsilon
    =-p_k^0\mathbf{I}+\varepsilon\mathbf{T}_k
    +O(\varepsilon^2),
\end{equation}
where $\mathbf{T}_k$ is given by \eqref{Cauchystress} in terms of the perturbation fields. 

Let $c_k^2 =(dp_k/d\rho_k)|_{\rho_k=\rho_k^0}$ denote the squared sound speed of phase $k$. Linearizing the first two lines of system \eqref{fullequation} yields
\[
    \partial_t p_k+c_k^2\rho_k^0\nabla\cdot\mathbf{v}_k=0,
    \qquad \rho_k^0\partial_t\mathbf{v}_k =\nabla\cdot\mathbf{T}_k.
\]
To linearize the interface condition, we write the perturbed interface in normal coordinates:
\[
    \Sigma_\varepsilon=\{x+\varepsilon\eta(x)\mathbf{n}(x):x\in\Sigma\},
\]
where $\eta$ denotes the normal displacement. The corresponding unit normal satisfies
\begin{equation}\label{variationnormal}
    \mathbf{n}_\varepsilon =\mathbf{n}-\varepsilon\nabla_\Sigma\eta +O(\varepsilon^2).
\end{equation}
With the convention $H=-\frac12\nabla_\Sigma\cdot\mathbf{n}$, the first variation of the mean curvature gives
\begin{equation}\label{perturbH}
    2H_\varepsilon =2H+\varepsilon(\Delta_\Sigma+|B|^2)\eta +O(\varepsilon^2);
\end{equation}
see \cite{huisken1984flow}. Here, $\Delta_\Sigma$ denotes the Laplace--Beltrami operator on $\Sigma$, and $|B|$ is the norm of the second fundamental form of $\Sigma$, \emph{i.e.},
\begin{equation}\label{defofB2}
  |B|^2 := \sum_{i,j,s,r=1}^2 h_{is} h_{jr} g^{ij} g^{sr} ,
\end{equation}
where $g^{-1}=(g^{ij})$ is the inverse of the induced Riemannian metric on $\Sigma$ by the embedding $\iota:\Sigma\hookrightarrow \mathbb{R}^3$ and $h=(h_{ij})$ is the corresponding second fundamental form.

Substituting \eqref{perturbedfield}, \eqref{perturbT}, \eqref{variationnormal}, and \eqref{perturbH} into the interface conditions in system \eqref{fullequation} and using the equilibrium stress balance \eqref{boundmoment}, we obtain
\begin{equation}
    \partial_t\eta=\mathbf{v}_f\cdot\mathbf{n}=\mathbf{v}_g\cdot\mathbf{n}, 
    \qquad 
    (\mathbf{T}_g-\mathbf{T}_f)\mathbf{n}=\sigma(\Delta_\Sigma+|B|^2)\eta\,\mathbf{n}.
\end{equation}

The linearized system for $(\eta,p_f,\mathbf{v}_f,p_g,\mathbf{v}_g)$ reads
\begin{equation}\label{linearq}
\left\{
\begin{aligned}
    &\partial_t p_k+c_k^2\rho_k^0\nabla\cdot\mathbf{v}_k=0,&& \\
    &\rho_k^0\partial_t\mathbf{v}_k=\nabla\cdot\mathbf{T}_k,&& \\
    &\partial_t\eta=\mathbf{v}_f\cdot\mathbf{n}=\mathbf{v}_g\cdot\mathbf{n},&& \text{on }\Sigma,\\
    &(\mathbf{T}_g-\mathbf{T}_f)\mathbf{n}=\sigma(\Delta_\Sigma+|B|^2)\eta\,\mathbf{n},&& \text{on }\Sigma.
\end{aligned}
\right.
\end{equation}

We now consider time-harmonic solutions with the convention $\mathrm{e}^{\mathrm{i}\omega t}$:
\[
    p_k(t,x)=p_k(x)\mathrm{e}^{\mathrm{i}\omega t},
    \qquad \eta(t,x)=\eta(x)\mathrm{e}^{\mathrm{i}\omega t},
    \qquad \mathbf{v}_k(t,x) =\mathbf{v}_k(x)\mathrm{e}^{\mathrm{i}\omega t}.
\]
The kinematic condition $\partial_t\eta=\mathbf{v}_f\cdot\mathbf{n}$ then gives $\eta=(\mathrm{i}\omega)^{-1} \mathbf{v}_f\cdot\mathbf{n}$. Eliminating $\eta$ from \eqref{linearq}, we obtain the time-harmonic system
\begin{equation}\label{linearqharmonic}
\left\{
\begin{aligned}
    &\mathrm{i}\omega p_k+c_k^2\rho_k^0\nabla\cdot\mathbf{v}_k=0,\\
    &\mathrm{i}\omega\rho_k^0\mathbf{v}_k=\nabla\cdot\mathbf{T}_k,\\
    &\mathbf{v}_f\cdot\mathbf{n}=\mathbf{v}_g\cdot\mathbf{n},
    && \text{on }\Sigma,\\
    &(\mathbf{T}_g-\mathbf{T}_f)\mathbf{n}
    =\frac{\sigma}{\mathrm{i}\omega}
    (\Delta_\Sigma+|B|^2)
    (\mathbf{v}_f\cdot\mathbf{n})\mathbf{n},
    && \text{on }\Sigma.
\end{aligned}
\right.
\end{equation}
Neglecting viscosity in the gas phase, \emph{i.e.}, let $\lambda_g=\mu_g=0$, we recover the transmission problem \eqref{maineq} studied in this paper.

\begin{remark}
Indeed, at a stationary equilibrium with constant pressures and vanishing velocities, the stress balance condition \eqref{boundmoment} implies that the interface has constant mean curvature and hence must be spherical. Nevertheless, the linearized transmission problem \eqref{maineq} can be formulated and analyzed mathematically for general smooth domains $D$, as considered in this paper. Such a generalization is also physically motivated by more complex configurations, such as encapsulated bubbles, rotating bubbles, and other systems that may admit nonspherical equilibrium shapes. In particular, combining the results obtained in this paper with those in \cite{ammari2019encapsulated}, the effect of coating on the acoustic resonances can be easily evaluated as the thickness of the coating is much smaller than the typical size of the micro-bubble.     
\end{remark}

\section{Some properties of hydrodynamic layer potentials}\label{appenlayerpotential}

In this appendix, we recall some properties of hydrodynamic layer potentials in a self-contained manner. For more details, see \cite[§1.4]{ammari2015mathematical}. Recall that $\mathbb{C}^*=\mathbb{C}\setminus \{0\}$.

\subsection{Scalar Helmholtz layer potentials}

For $z\in \mathbb{C}$, define
\begin{equation}\label{fundHel}
    G^z(x) =g^z(|x|), \qquad g^z(r) = -\frac{\mathrm{e}^{\mathrm{i} z r}}{4\pi r} , \qquad r>0.
\end{equation}
The function $G^z$ is locally integrable on $\mathbb R^3$ and satisfies
\begin{equation}\label{diracdistributionforGz}
    (\Delta+z^2)G^z=\delta_0 \quad\text{in }\mathcal D'(\mathbb R^3),
\end{equation}
where $\delta_0$ denotes the Dirac distribution at the origin. 
Moreover,
\begin{equation}\label{outgoingforGz}
    G^z(x)=\left\{
    \begin{aligned}
        & - \frac{\mathrm{i}z}{4\pi} h^{(1)}_0(z|x|), && z\in \mathbb{C}^*, \\
        & -\frac{1}{4\pi|x|} , && z=0.
    \end{aligned}\right.
\end{equation}
It follows that $G^z$ is outgoing at $z$ in the sense of Section \ref{secradiationcind}.

\begin{lemma}\label{lemmaC1}
For $j=0,1,2,\cdots$, define
\[
    G_j(x):=g_j(|x|),\qquad g_j(r):=-\frac{\mathrm i^j}{4\pi j!}\,r^{j-1}, \qquad r>0.
\]
For every integer $M\geq0$, $z\in\mathbb C$, and $x\in\mathbb R^3\setminus\{0\}$,
\begin{equation}\label{expansionGz}
    G^z(x)=\sum_{j=0}^{M}z^jG_j(x)+R_M^z(x),
\end{equation}
where, for every multi-index $\alpha$ and every $R,L>0$, there exists $C=C(M,\alpha,R,L)>0$ such that
\begin{equation}\label{remainderGzderivatives}
    |\partial_x^\alpha R_M^z(x)|\leq C|z|^{M+1}|x|^{M-|\alpha|},\qquad |z|\leq R,\quad 0<|x|\leq L.
\end{equation}
\end{lemma}

% \[
%     R_M^z(x)=-\frac{(\mathrm i z)^{M+1}|x|^M}{4\pi M!}\int_0^1(1-t)^M\mathrm e^{\mathrm i tz|x|}\,dt.
% \]

\begin{proof}
Taylor's formula with integral remainder gives
\[
    R_M^z(x):=G^z(x)-\sum_{j=0}^{M}z^jG_j(x)=z^{M+1}r^M F_M(zr), \qquad r=|x|,
\]
where
\begin{equation}\label{definitionofFM}
    F_M(w):=-\frac{\mathrm i^{M+1}}{4\pi M!}\int_0^1(1-t)^M\mathrm e^{\mathrm i tw}\,dt
\end{equation}
is entire. Repeated differentiation yields
\[
    \partial_x^\alpha\bigl(r^M F_M(zr)\bigr)=r^{M-|\alpha|}
    \sum_{\ell=0}^{|\alpha|}(zr)^\ell F_M^{(\ell)}(zr)a_{\alpha,\ell}(x/r),
\]
where $a_{\alpha,\ell}$ are smooth on $\mathbb{S}^2$. For $|z|\leq R$ and $r\leq L$, the absolute value of the summation is bounded by a constant depending only on $M,\alpha,R,L$. Multiplying by $|z|^{M+1}$ proves \eqref{remainderGzderivatives}.
\end{proof}

Recall that $\partial D$ is smooth, and $\mathbf n$ denotes the outward unit normal to $D$. The subscripts $+$ and $-$ indicate traces from the exterior and interior, respectively. 

For $z\in\mathbb C$ and $\varphi \in C^\infty(\partial D)$, the scalar Helmholtz single-layer potential is
defined by
\begin{equation}
    S^z[\varphi](x) :=\int_{\partial D}G^z(x-y)\varphi(y)\,dS(y),\qquad x\in\mathbb R^3\setminus\partial D.
\end{equation}
By \eqref{diracdistributionforGz} and \eqref{outgoingforGz}, $S^z[\varphi]$ satisfies the equation $(\Delta+z^2)S^z[\varphi]=0$ in $\mathbb R^3\setminus\partial D$, and is outgoing at $z$. Due to the weak singularity of $G^z$ at origin, $S^z[\varphi]$ has a common trace
\begin{equation}
    S^z[\varphi]\big|_+=S^z[\varphi]\big|_- \qquad \mathrm{on} \ \partial D.
\end{equation}
We use the same notation $S^z$ for this boundary trace operator.

The normal derivative of the single-layer potential has the jump formula
\begin{equation}\label{jumpformula}
    \left. \frac{\partial S^z  [\varphi]}{\partial \mathbf{n}} \right|_\pm = \left(
    \pm \frac{1}{2}I + K^{z,*} \right)[\varphi],
\end{equation}
where the Neumann-Poincar\'{e} operator $K^{z,*}$ is defined by
\begin{equation}
    K^{z,*} [\varphi](x) := \mathrm{p.v.} \int_{\partial D} \nabla G^z(x - y) \cdot\mathbf{n}(x) \varphi(y)\,dS(y) , \qquad x\in \partial D.
\end{equation}
Throughout, $\mathrm{p.v.}$ denotes the Cauchy principal value.

\begin{lemma}\label{lemmaC2}
For every $s\in\mathbb R$, the following hold.
\begin{itemize}
    \item[(a)] Both $S^z$ and $K^{z,*}$ extend to holomorphic $\mathcal{L}(H^s(\partial D),H^{s+1}(\partial D))$-valued functions over $\mathbb{C}$. 
    \item[(b)] For $\varphi \in C^\infty(\partial D)$ and $j\in \mathbb{N}$, define 
    \begin{equation}
    \begin{aligned}
    S_j[\varphi](x)&:=\int_{\partial D} G_j(x-y)\varphi(y)\,dS(y),\\
    K_j^*[\varphi](x)&:=\operatorname{p.v.}\int_{\partial D} \nabla G_j(x-y)\cdot\mathbf n(x)\,\varphi(y)\,dS(y),
    \end{aligned} \qquad x\in \partial D.
    \end{equation}
    Both $S_j$ and $K_j^*$ extend to bounded operators from $H^s(\partial D)$ to $H^{s+j+1}(\partial D)$. Here $K_1^*=0$.
\item[(c)] For every integer $M\geq 0$ and every $R>0$, there exists $C=C(s,M,R,D)>0$ such that
    \begin{equation}\label{Helmholtzoperatorremainder}
    \begin{aligned}
        &\left\|S^z-\sum_{j=0}^{M}z^jS_j\right\|_{\mathcal L(H^s(\partial D),H^{s+M+2}(\partial D))} \\
       &\quad +\left\|K^{z,*}-\sum_{j=0}^{M}z^jK_j^*\right\|_{\mathcal L(H^s(\partial D),H^{s+M+2}(\partial D))}  \leq C|z|^{M+1},\qquad |z|\leq R.
    \end{aligned}
    \end{equation}
\end{itemize}
\end{lemma}

\begin{proof}
We use the Sobolev mapping theorem for pseudodifferential operators \cite[Theorem~18.1.13]{HormanderIII}: an operator of order $-m$ on $\partial D$ is bounded from $H^s(\partial D)$
to $H^{s+m}(\partial D)$, with the operator norm controlled by finitely many symbol seminorms. The operator orders below are obtained by taking the Fourier transform of the kernels in local charts; see \cite[§18.2]{HormanderIII}.

Write $r=|x-y|$ and $b(x,y)=(x-y)\cdot\mathbf n(x)$. The kernels of $S_j$ and $K_j^*$ are
\[
    a_j(x,y)=-\frac{\mathrm i^j}{4\pi j!}r^{j-1},    \qquad    
    b_j(x,y)=-\frac{\mathrm i^j(j-1)}{4\pi j!}b(x,y)r^{j-3},
\]
respectively. Smoothness of $\partial D$ implies that $b$ vanishes quadratically along the diagonal. Consequently, both kernels have homogeneous expansions beginning at degree $j-1$. Since $\dim\partial D=2$, their Fourier transforms in local charts are symbols of order at most $-(j+1)$. The Sobolev mapping theorem therefore proves \textup{(b)}. Also, $K_1^*=0$ because $G_1$ is constant.

By Lemma \ref{lemmaC1}, write
\[
    R_M^z(x-y)=z^{M+1}r^M F_M(zr),
\]
where $F_M$ defined in \eqref{definitionofFM} is entire. Hence, the kernels of the remainders
\begin{equation}\label{normalizedoperators}
    \frac{1}{z^{M+1}} \left( S^z-\sum_{j=0}^{M}z^jS_j \right), \qquad \frac{1}{z^{M+1}} \left( K^{z,*}-\sum_{j=0}^{M}z^jK_j^* \right)
\end{equation}
are
\begin{equation}\label{normalizedoperatorskernels}
    r^M F_M(zr), \qquad b(x,y)r^{M-2} \bigl(MF_M(zr)+zrF_M'(zr)\bigr),
\end{equation}
respectively. Their homogeneous expansions begin at degree $M$ or higher, so the operators in \eqref{normalizedoperators} have order at most $-(M+2)$. The estimate \eqref{remainderGzderivatives}, together with the quadratic vanishing of $b$, gives uniform bounds for the associated seminorms of symbols \eqref{normalizedoperatorskernels} when $|z|\leq R$. The Sobolev mapping theorem therefore gives uniform bounds for operators in \eqref{normalizedoperators} from $H^s(\partial D)$ to $H^{s+M+2}(\partial D)$. Multiplication by $|z|^{M+1}$ proves \textup{(c)}.

Finally, the kernels in \eqref{normalizedoperatorskernels} are entire in $z$, and their $z$-derivatives satisfy the same symbol bounds uniformly on compact subsets of $\mathbb C$. Therefore, the Sobolev mapping theorem implies holomorphy in the $\mathcal L(H^s(\partial D),H^{s+M+2}(\partial D))$ norm. Taking $M=0$, and using that $S_0$ and $K_0^*$ are bounded from $H^s(\partial D)$ to $H^{s+1}(\partial D)$, prove \textup{(a)}. 
\end{proof}

\subsection{Hydrodynamic layer potentials}

For $k_p\in \mathbb{C}$ and $k_s \in \mathbb{C}^*$, define
the hydrodynamic kernels $(\Pi^{k_p,k_s},\mathbf{\Gamma}^{k_p,k_s})$ by
\begin{equation}\label{deps}
     \Pi^{k_p,k_s} := \left( 2\frac{k_p^2}{k_s^2} -1 \right) \nabla G^{k_p}, \qquad \mathbf{\Gamma}^{k_p,k_s} := \mathbf{\Gamma}_p^{k_p,k_s}  + \mathbf{\Gamma}_s^{k_s},
\end{equation}
where the compressional component $\mathbf{\Gamma}_p^{k_p,k_s}  $ and the shear component $ \mathbf{\Gamma}_s^{k_s}$ are given by
\begin{equation}\label{depsd}
    \mathbf{\Gamma}_p^{k_p,k_s} : = -\frac{1}{k_s^2} \nabla \nabla G^{k_p}, \qquad \mathbf{\Gamma}_s^{k_s} := \frac{1}{k_s^2} \left(k_s^2\mathbf{I} +\nabla\nabla \right) G^{k_s}  ,
\end{equation}
respectively. Thus, $\mathbf{\Gamma}_p^{k_p,k_s}$ and $\mathbf{\Gamma}_s^{k_s}$ are outgoing at $k_p$ and $k_s$, respectively. The leading singularities of the two Hessians cancel, so the total velocity kernel $\mathbf{\Gamma}^{k_p,k_s}$ is locally integrable. For $2k_p^2\neq k_s^2$, these kernels satisfy 
\begin{equation}\label{funds}
\begin{cases}
    \nabla \cdot (2\nabla^{\mathrm{s}} \mathbf{\Gamma}^{k_p,k_s}  - \Pi^{k_p,k_s} \otimes  \mathbf{I}) +k_s^2 \mathbf{\Gamma}^{k_p,k_s}   = \delta_0 \mathbf{I}, \\
    \nabla \cdot \mathbf{\Gamma}^{k_p,k_s}  = \frac{k_p^2}{2k_p^2 - k_s^2}\Pi^{k_p,k_s} ,
\end{cases} \quad \mathrm{in} \ \mathcal D'(\mathbb R^3).
\end{equation}
% For fixed $k_p\in \mathbb{C}$, $k_s \in \mathbb{C}\setminus \{0\}$, and $r=|x|\geq1$, the explicit formulas give
% \[
%     |\mathbf\Gamma^{k_p,k_s}(x)|+|\Pi^{k_p,k_s}(x)|\leq
%     \frac{C(k_p,k_s)}r\left(\mathrm e^{-r\operatorname{Im}k_p}+\mathrm e^{-r\operatorname{Im}k_s}\right).
% \]
% In particular, the kernels decay exponentially when
% both imaginary parts are strictly positive.

For $\bm{\varphi} \in C^\infty(\partial D;\mathbb{C}^3)$, define the hydrodynamic single-layer potential operators by
\begin{equation}
\begin{aligned}
& Q^{k_p,k_s}[\bm{\varphi}] (x) := \int_{\partial D} \Pi^{k_p,k_s} (x-y) \cdot \bm{\varphi}(y)\,dS(y), \\
& \mathbf{S}^{k_p,k_s}[\bm{\varphi}] (x) := \int_{\partial D} \mathbf{\Gamma}^{k_p,k_s} (x-y) \bm{\varphi}(y)\,dS(y),  
\end{aligned}
    \qquad   x\in \mathbb{R}^3\setminus\partial D .
\end{equation}
The velocity single-layer potential is continuous across $\partial D$:
\[
    \mathbf S^{k_p,k_s}[\bm\varphi]\big|_+=\mathbf S^{k_p,k_s}[\bm\varphi]\big|_- \qquad\text{on }\partial D.
\]
We use the same notation $\mathbf S^{k_p,k_s}$ for this common boundary trace operator.

Set $q=Q^{k_p,k_s}[\bm\varphi]$ and $\mathbf v=\mathbf S^{k_p,k_s}[\bm\varphi]$. By \eqref{funds}, for $k_s\neq0$ and $2k_p^2\neq k_s^2$, the pair $(q,\mathbf v)$ is smooth outside $\partial D$ and satisfies
\begin{equation}\label{jumpform}
    \begin{cases}
    \nabla\cdot(2\nabla^{\mathrm{s}}\mathbf v-q\mathbf I)+k_s^2\mathbf v=0, \quad \nabla\cdot\mathbf v=\frac{k_p^2}{2k_p^2-k_s^2}\,q, &\text{in }\mathbb R^3\setminus\partial D, \\
    (2\nabla^{\mathrm{s}} \mathbf{v} - q \mathbf{I} ) \mathbf{n} \big|_\pm = \left(
    \pm \mathbf{I} /2 + \mathbf{K}^{k_p,k_s,*} \right)[\bm{\varphi}] & \mathrm{on}\ \partial D.
\end{cases}
\end{equation}
Here, the Neumann-Poincar\'{e} operator $\mathbf{K}^{k_p,k_s,*}$ is defined by
\begin{equation}
\begin{aligned}
    \mathbf{K}^{k_p,k_s,*}[\bm{\varphi}](x) &:= \mathrm{p.v.} \int_{\partial D} \Big(2\nabla^{\mathrm{s}}_x \big(\bm{\Gamma}^{k_p,k_s} (x-y) \bm{\varphi}(y)\big) \\
    & \quad \quad \quad \quad \quad \quad - \big( \Pi^{k_p,k_s}(x-y)\cdot \bm{\varphi}(y)\big) \mathbf{I} \Big) \mathbf{n} (x)\,dS(y) ,\qquad x\in \partial D.
\end{aligned}
\end{equation}

\begin{remark}\label{outgoingforhydrolayerpotential}
    By \eqref{deps}--\eqref{depsd}, we have
\[
\begin{aligned}
    q&=Q^{k_p,k_s}[\bm\varphi]=\frac{2k_p^2-k_s^2}{k_s^2}\int_{\partial D} \nabla G^{k_p}(\cdot-y)\cdot\bm\varphi(y)\,dS(y),\\
    \mathbf v_s&=\mathbf v-\frac{1}{k_s^2-2k_p^2}\nabla q
      =\int_{\partial D} \mathbf\Gamma_s^{k_s}(\cdot-y)\bm\varphi(y)\,dS(y).
\end{aligned}
\]
Thus, $q$ is outgoing at $k_p$ for every $k_p\in \mathbb{C}$, and $\mathbf v_s$ is outgoing at $k_s$ for every $k_s\in \mathbb{C}^*$.
\end{remark}

\begin{lemma}\label{lemmaKmaping} 
Denote $H^s=H^s(\partial D;\mathbb C^3)$ and $\vartheta=k_p^2/k_s^2$. For every $s\in\mathbb R$, the following hold.
\begin{itemize}
    \item[(a)] The families
    \[
    \begin{aligned}
        (k_p,k_s)&\mapsto\mathbf S^{k_p,k_s} \in\mathcal L( H^s, H^{s+1}),\\
        (k_p,k_s)&\mapsto\mathbf K^{k_p,k_s,*} \in\mathcal L( H^s)
    \end{aligned}
    \]
    are holomorphic on $\mathbb C\times \mathbb C^*$. At $k_p=0$, 
    \begin{equation}\label{C20bounded}
        k_s\mapsto\mathbf K^{0,k_s,*}\in\mathcal L( H^s, H^{s+1})
    \end{equation}
    is holomorphic on $\mathbb C^*$.
    \item[(b)] Define $\mathbf\Gamma_0^{k_s}:=\mathbf\Gamma^{0,k_s}$, $\Pi_0^{k_s}:=\Pi^{0,k_s}$,
    and, for $j\geq1$,
    \[
        \mathbf\Gamma_j^{k_s}
        :=-\frac1{k_s^2}\nabla\nabla G_j,
        \qquad
        \Pi_j^{k_s}
        :=-\nabla G_j+\frac2{k_s^2}\nabla G_{j-2},
    \]
    where $G_{-1}:=0$.
    For $\bm\varphi \in C^\infty(\partial D;\mathbb{C}^3)$, define
    \[
        \mathbf S_j^{k_s}[\bm\varphi](x):=\int_\Gamma\mathbf\Gamma_j^{k_s}(x-y)\bm\varphi(y)\,dS(y), \qquad x\in \partial D,
    \]
    and
    \begin{equation}\label{defofKjksstar}
        \begin{aligned}
        \mathbf K_j^{k_s,*}[\bm\varphi](x):=\operatorname{p.v.}\int_{\partial D}
        \Bigl(&2\nabla_x^{\mathrm s}\bigl(\mathbf\Gamma_j^{k_s}(x-y)\bm\varphi(y)\bigr)\\
            &-\bigl(\Pi_j^{k_s}(x-y)\cdot\bm\varphi(y)\bigr)\mathbf I\Bigr)\mathbf n(x)\,dS(y), \qquad x\in \partial D.
    \end{aligned}
    \end{equation}
    Then
    \[
        \mathbf S_0^{k_s}=\mathbf S^{0,k_s}, \qquad
        \mathbf K_0^{k_s,*}=\mathbf K^{0,k_s,*}, \qquad \mathbf S_1^{k_s}=\mathbf K_1^{k_s,*}=0.
    \]
    For every $j\geq2$, these operators extend to bounded operators
    \[
        \mathbf S_j^{k_s}: H^s\rightarrow H^{s+j-1}, \qquad \mathbf K_j^{k_s,*}: H^s\rightarrow H^{s+j-2}.
    \]
    \item[(c)]
    For every integer $M\geq1$, every $R>0$, and every compact set $\mathcal C\subset\mathbb C^*$,
    there exists $C=C(s,M,R,D,\mathcal C)>0$ such that
    \begin{equation}\label{hydrodynamicOperatorRemainder}
\begin{aligned}
&\left\|\mathbf S^{k_p,k_s}-\sum_{j=0}^{M}k_p^j\mathbf S_j^{k_s}\right\|_{\mathcal L( H^s, H^{s+M})} \\
&\quad+ \left\|\mathbf K^{k_p,k_s,*}-\sum_{j=0}^{M}k_p^j\mathbf K_j^{k_s,*} \right\|_{\mathcal L(H^s,H^{s+M-1})} \leq C|k_p|^{M+1},
\end{aligned}
\end{equation}
for $|k_p|\leq R$ and $k_s\in\mathcal C$. 
    \item[(d)] For
$(x,\xi)\in T^*\partial D\setminus0$, set $\mathbf n=\mathbf n(x)$ and $e=\xi^\sharp/|\xi|_g$,
where $g$ is the induced metric on $\partial D$. The operators $\mathbf S^{k_p,k_s}$ and
$\mathbf K^{k_p,k_s,*}$ are pseudodifferential operators of orders $-1$ and at most $0$, respectively, 
with principal symbols
\begin{equation}\label{hydrodynamicPrincipalSymbols}
\begin{aligned}
    \sigma_{-1}(\mathbf S^{k_p,k_s})(x,\xi)&=-\frac{1}{4|\xi|_g}
      \left[2\mathbf I-(1-\vartheta)(e\otimes e+\mathbf n\otimes\mathbf n)\right],\\
    \sigma_0(\mathbf K^{k_p,k_s,*})(x,\xi)&=-\frac{\mathrm i\vartheta}{2}(e\otimes\mathbf n-\mathbf n\otimes e).
\end{aligned}
\end{equation}
% In particular,
% \begin{equation}\label{hydrodynamicNormalSymbol}
%     \mathbf n^\top\sigma_{-1}(\mathbf S^{k_p,k_s})(x,\xi)=-\frac{1+\vartheta}{4|\xi|_g}\,\mathbf n^\top.
% \end{equation}
% If $|\vartheta|<1/2$, then
% $L:=\tfrac12\mathbf I+\mathbf K^{k_p,k_s,*}$ is elliptic,
% and its principal symbol $a_0:=\sigma_0(L)$ satisfies
% \begin{equation}\label{tractionInverseNormalSymbol}
%     \mathbf n^\top a_0(x,\xi)^{-1}\mathbf n
%     =\frac{2}{1-\vartheta^2}.
% \end{equation}
% Moreover, the scalar operator
% $\mathcal V_{k_s}[f]:=\mathbf n\cdot\mathbf S^{0,k_s}[f\mathbf n]$
% is elliptic of order $-1$, with
% \begin{equation}\label{normalSingleLayerSymbol}
%     \sigma_{-1}(\mathcal V_{k_s})(x,\xi)
%     =-\frac{1}{4|\xi|_g}.
% \end{equation}
    % \item[(e)] For $\bm\varphi\in C^\infty(\partial D;\mathbb C^3)$and $\operatorname{Im}k_s>0$,
    % \[
    % \begin{aligned}
    %     |\mathbf S^{0,k_s}[\bm\varphi](x)|+r|\nabla\mathbf S^{0,k_s}[\bm\varphi](x)|&=O(r^{-3}),\\
    %     |Q^{0,k_s}[\bm\varphi](x)|+r|\nabla Q^{0,k_s}[\bm\varphi](x)|&=O(r^{-2}),
    % \end{aligned} \qquad r=|x|\to\infty.
    % \]
\end{itemize}
\end{lemma}

\begin{proof}
Write $\mathbf r=x-y$ and $r=|\mathbf r|$. Lemma \ref{lemmaC1} gives
\begin{equation}\label{hydroLocalExpansion}
\begin{aligned}
    \mathbf\Gamma^{k_p,k_s}(\mathbf r)={}&
    -\frac1{8\pi}\left[(1+\vartheta)\frac{\mathbf I}{r}+(1-\vartheta)\frac{\mathbf r\otimes\mathbf r}{r^3}
    \right]+\mathbf C(k_p,k_s)+E_{\bm \Gamma},\\
    \Pi^{k_p,k_s}(\mathbf r)={}&\frac{2\vartheta-1}{4\pi}\frac{\mathbf r}{r^3}+E_\Pi,
\end{aligned}
\end{equation}
where $\mathbf C(k_p,k_s)$ is independent of $\mathbf r$. If $\boldsymbol\zeta=(k_p,k_s)$ ranges over a compact subset of $\mathbb{C}\times \mathbb{C}^*$, then
\begin{equation}\label{derivativetozeta}
    |\partial_{\mathbf r}^\alpha\partial_{\boldsymbol\zeta}^\beta E_{\bm \Gamma}|\leq C_{\alpha,\beta}r^{1-|\alpha|}, \qquad |\partial_{\mathbf r}^\alpha \partial_{\boldsymbol\zeta}^\beta E_\Pi| \leq C_{\alpha,\beta}r^{-|\alpha|} , \qquad 0<r\leq1,\  |\beta|\leq2.
\end{equation}
These estimates follow by \eqref{remainderGzderivatives}. Let $\mathscr T^{k_p,k_s}(\mathbf r,\mathbf n)$ denote the kernel of $\mathbf{K}^{k_p,k_s,*}$, \emph{i.e.},
\[
\mathscr T^{k_p,k_s}(\mathbf r,\mathbf n)\mathbf a:=
\left[2\nabla_{\mathbf r}^{\mathrm s}\bigl(\mathbf\Gamma^{k_p,k_s}(\mathbf r)\mathbf a\bigr)
-\bigl(\Pi^{k_p,k_s}(\mathbf r)\cdot\mathbf a\bigr)\mathbf I\right]\mathbf n.
\]
Differentiating \eqref{hydroLocalExpansion} and using \eqref{derivativetozeta} yield
\begin{equation}\label{hydroTractionExpansion}
\begin{aligned}
    \mathscr T^{k_p,k_s}(\mathbf r,\mathbf n) ={}&
    \frac{\vartheta}{4\pi r^3} \left[ (\mathbf r\cdot\mathbf n)\mathbf I +\mathbf r\otimes\mathbf n
        -\mathbf n\otimes\mathbf r \right]\\
    &+ \frac{3(1-\vartheta)}{4\pi} \frac{(\mathbf r\cdot\mathbf n)\mathbf r\otimes\mathbf r}{r^5} +O(1).
\end{aligned}
\end{equation}
On $\partial D$, $\mathbf r\cdot\mathbf n(x)=O(r^2)$. Thus, terms on the right-hand side of \eqref{hydroTractionExpansion} containing $\mathbf r\cdot\mathbf n$ have singularity at most $r^{-1}$. The remaining term,
\begin{equation}\label{degreeminus2}
    \frac{\vartheta}{4\pi} \frac{\mathbf r\otimes\mathbf n(x) -\mathbf n(x)\otimes\mathbf r}{r^3},
\end{equation}
is a homogeneous kernel of degree $-2$. Since $\dim \partial D =2$, \eqref{hydroLocalExpansion} and \eqref{hydroTractionExpansion} give operator orders at most $-1$ and $0$ for $\mathbf S^{k_p,k_s}$ and $\mathbf K^{k_p,k_s,*}$, respectively. At $k_p=0$, the term \eqref{degreeminus2} vanishes, so $\mathbf K^{0,k_s,*}$ has the order of most $-1$. The Sobolev mapping theorem gives the boundedness in (a). 

The same order bounds hold for the first two derivatives of the kernels with respect to $\boldsymbol\zeta$; see \eqref{derivativetozeta}. This, combined with the Sobolev mapping theorem, implies that 
\[
    \left\|A(\boldsymbol\zeta+\boldsymbol\eta)-A(\boldsymbol\zeta)-\bm{\eta} \cdot  A^{(1)}(\boldsymbol\zeta)
    \right\|_{\mathscr X_s}\leq C|\boldsymbol\eta|^2,
\]
where $A(k_p,k_s)$ is either $\mathbf S^{k_p,k_s}$ or $\mathbf K^{k_p,k_s,*}$, $A^{(1)}$ is obtained by differentiating
its kernel, and $\mathscr X_s$ is respectively $\mathcal L( H^s, H^{s+1})$ or
$\mathcal L( H^s)$. This proves joint holomorphy. At $k_p=0$,  differentiation in $k_s$
preserves the order $-1$ of $\mathbf K^{0,k_s,*}$. Thus, \eqref{C20bounded} holds. This proves \textup{(a)}.

A direct computation yields that $\nabla\nabla G_j$ is homogeneous of degree $j-3$, and the kernel of $\mathbf{K}_j^{k_s,*}$ has leading degree $j-4$. The Sobolev mapping theorem therefore proves \textup{(b)}.

For $M\geq 2$, Lemma \ref{lemmaC1} gives
\[
\begin{aligned}
    \mathbf\Gamma^{k_p,k_s}-\sum_{j=0}^{M}k_p^j\mathbf\Gamma_j^{k_s}&=-\frac1{k_s^2}\nabla\nabla R_M^{k_p},\\
    \Pi^{k_p,k_s}-\sum_{j=0}^{M}k_p^j\Pi_j^{k_s}&=-\nabla R_M^{k_p}+\frac{2k_p^2}{k_s^2}\nabla R_{M-2}^{k_p}.
\end{aligned}
\]
Following the proof of Lemma \ref{lemmaC2}(c), the estimate \eqref{remainderGzderivatives} gives the desired uniform kernel bounds for $|k_p|\leq R$ and $k_s\in\mathcal C$. The case $M=1$ follows from \textup{(a)} and Taylor's formula in $k_p$, since
\[
    \left.\partial_{k_p}\mathbf S^{k_p,k_s}\right|_{k_p=0}=\left.\partial_{k_p}\mathbf K^{k_p,k_s,*}\right|_{k_p=0} =0.
\]
The vanishing derivatives follow directly from the kernel formulas and the fact that $G_1$ is constant. The estimate is uniform for $k_s$ in compact subsets of $\mathbb C^*$. The Sobolev mapping theorem then proves \eqref{hydrodynamicOperatorRemainder}.

To prove \textup{(d)}, fix $x\in\partial D$ and identify $T_x\partial D$ with $\mathbb R^2$ using an orthonormal basis. The principal kernels are the tangent-plane restrictions of the leading homogeneous terms in \eqref{hydroLocalExpansion} and \eqref{degreeminus2}; curvature and surface-measure corrections
contribute only to lower orders.

With the Fourier convention $\mathcal{F} f(\xi)=\int_{\mathbb R^2}e^{-\mathrm i y\cdot\xi}f(y)\,dy$, we have, for $\xi\ne0$,
\[
\begin{aligned}
    \mathcal F\bigl[|y|^{-1}\bigr]&=\frac{2\pi}{|\xi|},\\
    \mathcal F\bigl[\operatorname{p.v.}(y|y|^{-3})\bigr]&=-(2\pi\mathrm i)e,\\
    \mathcal F\bigl[y\otimes y\,|y|^{-3}\bigr]&=\frac{2\pi}{|\xi|}
      \bigl(\mathbf I-\mathbf n\otimes\mathbf n-e\otimes e\bigr).
\end{aligned}
\]
Here, tangent vectors are viewed as vectors in $\mathbb R^3$. The last two identities follow by distributional differentiation from the first, using
\[
    \frac{y_j}{|y|^3}=-\partial_j|y|^{-1}, \qquad
    \frac{y_i y_j}{|y|^3} =\frac{\delta_{ij}}{|y|} -\partial_j\!\left(\frac{y_i}{|y|}\right), \qquad i,j=1,2.
\]
Substitution into the principal kernels gives \eqref{hydrodynamicPrincipalSymbols}.
\end{proof}

\section{Proof of Proposition~\ref{prop:gohberg-sigal}}
\label{appenxgstheory}

\noindent\textit{Proof of \textup{(i)}.}
Fix a characteristic value $z_0$, and write $A_0=A(z_0)$. Since $A_0$ is Fredholm of index zero, its range is closed and
\[
    r:=\dim\ker A_0=\operatorname{codim}\operatorname{ran}A_0<\infty.
\]
Choose closed complementary subspaces $X_1\subset X$ and $Y_0\subset Y$ such that
\[
    X=\ker A_0\oplus X_1, \qquad Y=Y_0\oplus\operatorname{ran}A_0.
\]
After identifying $\ker A_0$ and $Y_0$ with $\mathbb C^r$, write
\[
    A(z)=\begin{pmatrix}a(z)&b(z)\\ c(z)&d(z)
    \end{pmatrix},
\]
where
\[
\begin{aligned}
    a(z)&:\mathbb C^r\to\mathbb C^r,& b(z)&:X_1\to\mathbb C^r,\\
    c(z)&:\mathbb C^r\to\operatorname{ran}A_0,& d(z)&:X_1\to\operatorname{ran}A_0.
\end{aligned}
\]

At $z=z_0$, $a(z_0)=b(z_0)=c(z_0)=0$, and $d(z_0)$ is an isomorphism. Thus, $d(z)^{-1}$ exists and is holomorphic
near $z_0$. A simple computation gives
\begin{equation}\label{holomortransformofA}
    \begin{pmatrix}
        I&-b(z)d(z)^{-1}\\ 0&d(z)^{-1}
    \end{pmatrix}
    A(z)
    \begin{pmatrix}
        I&0\\ -d(z)^{-1}c(z)&I
    \end{pmatrix}
    = \begin{pmatrix}
        S(z)&0\\
        0&I_{X_1}
    \end{pmatrix},
\end{equation}
where
\[
    S(z):=a(z)-b(z)d(z)^{-1}c(z).
\]
The right factor on the left-hand side in \eqref{holomortransformofA} is an isomorphism from
$\mathbb C^r\oplus\operatorname{ran}A_0$ to $\mathbb C^r\oplus X_1$, and the third matrix is an
automorphism of $\mathbb C^r\oplus X_1$. Both matrices and their inverses are holomorphic.
In particular, $S(z)$ is an $r\times r$ holomorphic matrix with $S(z_0)=0$, and $A(z)$ is invertible exactly when
$S(z)$ is invertible. In particular, $\det S$ is not identically zero near $z_0$ by the analytic Fredholm theorem.

By permutations, one may assume that $s_{11}$ has the smallest order of vanishing at $z_0$, denoted by $m_1\geq1$. Thus,
\[
    s_{11}(z)=(z-z_0)^{m_1}h(z),
    \qquad h(z_0)\neq0.
\]
Every entry of $S$ is divisible by $s_{11}(z)$ near a neighborhood of $z_0$ such that $h$ is nonvanishing.
Therefore, there is a holomorphic matrix $S_1(z)$ with $S_1(z_0)=0$ such that
\[
    S(z)\sim \begin{pmatrix}
        (z-z_0)^{m_1}&0\\
        0&S_1(z)
    \end{pmatrix}.
\]
Here and below, $\sim$ denotes equivalence under holomorphic invertible left and right transformations.
Moreover, $\det S_1$ is not identically zero. Repeating the same procedure yields
\[
    S(z)\sim\operatorname{diag}((z-z_0)^{m_1},\cdots,(z-z_0)^{m_r}), \qquad m_j\geq1.
\]
Together with \eqref{holomortransformofA}, this constructs
the isomorphisms $E(z)$ and $G(z)$ in \eqref{eq:GS-factorization}.

To prove the uniqueness of $m_j$ up to permutations, define, for each integer $n\geq1$, 
\[
    V_n(A):=\left\{u_0\in\ker A(z_0):
    \begin{array}{l}
        \text{there exists a holomorphic }u(z)\text{ near }z_0\text{ such that}\\
        u(z_0)=u_0 \text{ and }A(z)u(z)=\mathcal O_Y(|z-z_0|^n)
    \end{array}
    \right\}.
\]
Let $D(z)=\operatorname{diag}((z-z_0)^{m_1},\cdots,(z-z_0)^{m_r},I_{X_1})$.
Since $D=EAG$, holomorphic invertibility gives
\[
    V_n(A)=G(z_0)V_n(D).
\]
It is clear that $V_n(D)$ is spanned by the coordinate vectors whose exponents satisfy $m_j\geq n$.
Hence,
\begin{equation}\label{GSfiltrationFormula}
    \dim V_n(A)=\#\{j:m_j\geq n\},  \qquad  \#\{j:m_j=n\} =\dim V_n(A)-\dim V_{n+1}(A).
\end{equation}
Thus, $m_j$ are unique up to permutation. This proves \textup{(i)}.

\medskip \noindent\textit{Proof of \textup{(ii)}.}
Apply $V_n(A)$ above with $z_0=0$. For every $u_0\in\ker A_0$, 
\[
    A(z)u_0 =z^2A_2u_0+\mathcal O_Y(|z|^3).
\]
Hence, $V_2(A)=\ker A_0$, and
\eqref{GSfiltrationFormula} shows that all $m_j \geq 2$. Introduce the coordinate maps
\[
    U:\mathbb C^r\to X, \qquad U\alpha:=\sum_{j=1}^r\alpha_j u_j,
    \qquad L:Y\to\mathbb C^r, \qquad Lf:=(\ell_1(f),\cdots,\ell_r(f))^\top.
\]
Then $U$ is injective, $L$ is surjective, $\ker L=\operatorname{ran}A_0$, and $M=LA_2U$.
If $u_0\in V_3(A)$, there is a holomorphic function
\[
    u(z)=u_0+zv_1+z^2v_2+\mathcal O_X(|z|^3)
\]
satisfying $A(z)u(z)=\mathcal O_Y(|z|^3)$. The coefficient of $z^2$ therefore gives
\[
    A_0v_2=-A_2u_0.
\]
Writing $u_0=U\alpha$ and applying $L$ yields $M\alpha=0$. The invertibility of $M$ implies $u_0=0$. Thus, $V_3(A)=\{0\}$, and \eqref{GSfiltrationFormula} implies that every $m_j=2$.

The factorization in \textup{(i)} now has the form
\[
    E(z)A(z)G(z)=\operatorname{diag}(z^2I_r,I_{X_1}).
\]
Let $\Pi$ be the projection of $\mathbb C^r\oplus X_1$ onto $\mathbb C^r\oplus\{0\}$. For $z\neq0$,
\[
    A(z)^{-1}=z^{-2}G(z)\Pi E(z)+G(z)(I-\Pi)E(z).
\]
It follows that
\begin{equation}\label{expansionofAinverse}
    A(z)^{-1}=\frac{P_{-2}}{z^2}+\frac{P_{-1}}{z}+\mathcal O_{\mathcal L(Y,X)}(1),
    \qquad P_{-2}=G(0)\Pi E(0),
\end{equation}
and the rank of $P_{-2}$ is $r$. Comparing the coefficients of $z^{-2}$ in
$A(z)A(z)^{-1}=I_Y$ and $A(z)^{-1}A(z)=I_X$ gives
\[
    A_0P_{-2}=0, \qquad P_{-2}A_0=0.
\]
Thus, $P_{-2}$ takes values in $\operatorname{ran}U$ and vanishes on $\ker L$. There is therefore a unique
matrix $B\in\mathbb C^{r\times r}$ such that $P_{-2}=UBL$.
Finally, substituting
\[
    A(z)U =z^2A_2U +\mathcal O_{\mathcal L(\mathbb C^r,Y)}(|z|^3)
\]
and \eqref{expansionofAinverse} into $A(z)^{-1}A(z)U=U$ and letting $z\to0$ yields
\[
    P_{-2}A_2U=U, \qquad UBM=U.
\]
Since $U$ is injective, $BM=I_r$, and hence $B=M^{-1}$.
Consequently,
\[
    P_{-2}=UM^{-1}L
    =\sum_{i,j=1}^r(M^{-1})_{ij}\,u_i\otimes\ell_j.
\]
This proves \eqref{eq:GS-inverse} and \textup{(ii)}.

\medskip \noindent\textit{Proof of \textup{(iii)}.}
This is the generalized Rouch\'e theorem. See
\cite[Theorem~1.15]{ammari2018mathematical}.

\medskip \noindent\textit{Proof of \textup{(iv)}.}
Fix the complementary subspaces used in the block reduction at $(z_0,\eta_0)$.
The same reduction applies to $A(z,\eta)$ near $(z_0,\eta_0)$, and all blocks and eliminating factors
are jointly holomorphic. Since simplicity implies
$\dim\ker A(z_0,\eta_0)=1$, there is a scalar holomorphic function $s(z,\eta)$ such that
\[
    A(z,\eta)\sim\operatorname{diag}(s(z,\eta),I).
\]
At $\eta=\eta_0$, simplicity is equivalent to
\[
    s(z_0,\eta_0)=0, \qquad \partial_zs(z_0,\eta_0)\neq0.
\]
The holomorphic implicit-function theorem gives a unique holomorphic function $z(\eta)$ near $\eta_0$, with
$z(\eta_0)=z_0$, describing all zeros of $s$ in a sufficiently small neighborhood. After shrinking that neighborhood,
\[
    \partial_zs(z(\eta),\eta)\neq0.
\]
Thus, these zeros remain simple, and the same is true of the corresponding characteristic values of $A$.
This proves \textup{(iv)}.

\medskip \noindent\textit{Proof of \textup{(v)}.} This is the generalized logarithmic residue formula. See
\cite[Theorem~1.14]{ammari2018mathematical}.

\bibliographystyle{abbrv}
\bibliography{mybib_revised}

\end{document}